\documentclass[a4paper,11pt,reqno]{amsart}

\usepackage{esint} 

\usepackage{enumerate}
\usepackage{amsmath,amsthm,amssymb,amsfonts,enumitem,mathtools,mathrsfs,dsfont}
\usepackage[english]{babel}
\usepackage{epsf,epsfig}
\usepackage{hyperref}
\usepackage[margin=1in]{geometry}

\hypersetup{ colorlinks=true, linkcolor=blue, urlcolor=blue, citecolor=red }

\newtheorem{theorem}{Theorem}[section]
\newtheorem{proposition}[theorem]{Proposition}

\newtheorem{lemma}[theorem]{Lemma}
\newtheorem{definition}[theorem]{Definition}

\numberwithin{equation}{section}

\usepackage{accents}
\newcommand*{\dt}[1]{%
  \accentset{\mbox{\large\bfseries .}}{#1}}

\title[The Solution to Agmon's Conjecture for the polylaplacian]
{The Solution to Agmon's Conjecture for the polylaplacian in rough domains in $\mathbb{R}^{n}$}

\author{Artur Andrade}
\address{Artur Andrade
\\
Department of Mathematics
\\
Baylor University
\\
Sid Richardson Bldg., 1410 S.~4th Street
\\
Waco, TX 76706, USA} \email{ArturHenrique\_\,DeOliv@baylor.edu}

\author{Dorina Mitrea}
\address{Dorina Mitrea
\\
Department of Mathematics
\\
Baylor University
\\
Sid Richardson Bldg., 1410 S.~4th Street
\\
Waco, TX 76706, USA} \email{Dorina\_\,Mitrea@baylor.edu}

\author{Irina Mitrea}
\address{Irina Mitrea
\\
Department of Mathematics
\\
Temple University\!
\\
1805\,N.\,Broad\,Street
\\
Philadelphia, PA 19122, USA} \email{imitrea@temple.edu}

\author{Marius Mitrea\\ }
\address{Marius Mitrea
\\
Department of Mathematics
\\
Baylor University
\\
Sid Richardson Bldg., 1410 S.~4th Street
\\
Waco, TX 76706, USA} \email{Marius\_\,Mitrea@baylor.edu}


\date{\today}


\begin{document}

\begin{abstract}
In his 1957 paper on the Dirichlet problem for higher-order elliptic equations in the plane, Agmon conjectured that 
there should exist higher-order equations in dimensions greater than two for which a potential-theoretic approach based 
on multi-layer integral operators can be successfully implemented. We show that this is indeed the case for the polyharmonic 
operator $\Delta^m$ of arbitrary order $m\in\mathbb{N}$ in $\mathbb{R}^n$, $n\geq 2$, even in geometrically rough domains permitting 
singularities beyond the scope of the existing theory of higher-order elliptic boundary value problems. Central to our approach is a 
new genre of multi-layer singular integral operators, which play for $\Delta^m$ the role assumed by the classical harmonic double 
layer potential in the treatment of the Dirichlet problem for harmonic functions.
\end{abstract}

\maketitle

\section{Main Result}
\label{Sec:1}

The present work lies at the intersection of three longstanding research programs. First, in \cite[p.\,90]{Cal80}, A.P.~Calder\'on 
advocated extending boundary layer potentials ``{\it to much more general elliptic systems $[$than the Laplacian$]$}.'' 
Second, in \cite[Problem~7, p.\,xvii]{Riv}, N.M.~Rivi\`ere posed the problem of ``{\it identifying classes of boundary data that ensure 
existence and uniqueness $[$for boundary value problems for $\Delta^2$ on ${\mathscr{C}}^1$ domains$]$}.'' Third, in his landmark 1957 paper 
on the Dirichlet problem for scalar higher-order elliptic operators in the plane, S.~Agmon observed that ``{\it all these results indicate 
strongly the possibility of a general potential theory for higher-order equations}'' \cite[pp.\,180--181]{Ag}. 

Here we pursue these themes by solving the Dirichlet and Regularity problems for the polylaplacian $\Delta^m$, 
the prototypical higher-order elliptic operator. Our results apply to a broad class of sets, including certain non-Lipschitz domains, 
and accommodate boundary data belonging to a wide spectrum of spaces characterized primarily by their compatibility with 
harmonic-analytic techniques. With notation and terminology to be clarified shortly, the main theorem reads as follows.

\begin{theorem}[\underline{Higher-Order Dirichlet and Regularity Problems for the Polylaplacian}]\label{DInRBVP-tk-INTRO}
Suppose $\Omega\subseteq{\mathbb{R}}^n$, for $n\in{\mathbb{N}}$ with $n\geq 2$, is an Ahlfors regular domain. Abbreviate $\sigma:={\mathcal{H}}^{n-1}\lfloor\partial\Omega$ and denote by $\nu$ the outward unit normal vector to $\Omega$. Also, fix an arbitrary reference point $x_\ast\in{\mathbb{R}}^n\setminus\overline\Omega$, and pick an aperture parameter $\kappa\in(0,\infty)$. Finally, assume ${\mathbb{X}}$ is a normed function space belonging to the ``hierarchy of good spaces'' on $\partial\Omega$ and, having selected an arbitrary $m\in{\mathbb{N}}$, formulate the ${\mathbb{X}}$-Dirichlet Problem for $\Delta^m$ in $\Omega$ as follows{\rm :}
\begin{equation}\label{tk1-acxvtru-INTRO.bis}
\begin{array}{l}
{\mathbb{X}}\text{-}{\rm DP}\,\,
\left\{
\begin{array}{l}
u\in{\mathscr{C}}^\infty(\Omega),\,\,\,\Delta^m u=0\,\text{ in }\,\Omega,\,\,\,\mathcal{N}_\kappa(\nabla^{m-1} u)\in{\mathbb{X}},
\\[8pt]
(\partial^\gamma u)\big|^{{}^{\kappa\text{\rm -nt}}}_{\partial\Omega}=f_\gamma\,\,\text{ for each $\gamma\in{\mathbb{N}}_0^n$ with $|\gamma|\leq m-1$},
\\[2pt]
\text{for an arbitrary $\dot{f}=\{f_\gamma\}_{|\gamma|\leq m-1}\in{\rm HWA}_{m-1}[{\mathbb{X}}]$}.
\end{array}
\right.
\end{array}
\end{equation}

Then there exists a threshold $\delta\in(0,1)$, depending only on dimension $n$, the Ahlfors regularity character of $\partial\Omega$, 
the space ${\mathbb{X}}$, and $m$, such that if {\rm(}with {\rm BMO} denoting the John-Nirenberg space{\rm )}
\begin{equation}\label{bavbx638n-tk1-INTRO} 
\|\nu\|_{{\rm BMO}(\partial\Omega,\sigma)}<\delta
\end{equation}
one concludes that $\Omega$ is actually a two-sided {\rm NTA} domain, in the sense of Jerison-Kenig {\rm\cite{JeKe82}}, with an unbounded 
boundary, and the ${\mathbb{X}}$-Dirichlet Problem \eqref{tk1-acxvtru-INTRO.bis} has a unique solution. 

In addition, the unique solution $u$ of \eqref{tk1-acxvtru-INTRO.bis} satisfies the following properties{\rm :}

\begin{enumerate}[label=\textit{(\roman*)}, ref=\textit{(\roman*)}]
\item\label{property1-INTRO} There exists some constant $C\in(0,\infty)$ independent of $\dot{f}$ with the property that 
\begin{equation}\label{a538bxd9-tk2-INTRO.DIR}
\|{\mathcal{N}}_\kappa(\nabla^{m-1}u)\|_{\mathbb{X}}\leq C\|\dot{f}\|_{{\rm HWA}_{m-1}[{\mathbb{X}}]}. 
\end{equation}

\item\label{property2-INTRO} 
With ${\mathbb{P}}_{m-2,x_\ast}$ denoting the Taylor polynomial functor 
{\rm (}of order $m-2$ and base point $x_\ast${\rm ;} cf. \eqref{mxcvbprm3fdh.PP-INTRO}{\rm )},
at each point $x\in\Omega$ one may express the function $u$ as
\begin{align}\label{eq:U-INT-REP-FORM}
\qquad u(x)=\frac{1}{\omega_{n-1}}\sum_{|\gamma|\leq m-1}\,\,\,\int\limits_{\partial\Omega}\nu(y)\cdot\Bigg\{ &
\frac{y-x}{|x-y|^{n}}\frac{(x-y)^{\gamma}}{\gamma!}
\nonumber\\[-2pt]
&\quad
-{\mathbb{P}}_{m-2,x_\ast}\Bigg[\frac{y-\cdot}{|\cdot-y|^{n}}\frac{(\cdot-y)^{\gamma}}{\gamma!}\Bigg](x)\Bigg\}
g_{\gamma}(y)d\sigma(y),
\end{align}
for some $\dot{g}=\{g_{\gamma}\}_{|\gamma|\leq m-1}\in{\rm HWA}_{m-1}[{\mathbb{X}}]$ which is uniquely determined {\rm(}and quantitatively controlled{\rm)} by the boundary datum $\dot{f}$.
\vskip 0.08in

\item\label{property3-INTRO} At each point $x\in\Omega$, the vector of derivatives $\nabla^{m-1}u:=(\partial^{\lambda}u)_{|\lambda|=m-1}$ is represented as
\begin{equation}\label{MSafddfgdsa-THM}
\qquad
(\nabla^{m-1}u)(x)=\Bigg(\sum_{|\beta|=m-1}\int_{\partial\Omega}\nu(y)\cdot (y-x)\, k_{\lambda\beta}(x-y)g_{\beta}(y)d\sigma(y)\Bigg)_{|\lambda|=m-1}
\end{equation}
where 
\begin{equation}\label{HShagdfsdg-THM}
\parbox{10.70cm}{for each pair of multi-indices $\lambda,\beta\in\mathbb{N}_{0}^{n}$ with $|\lambda|=|\beta|=m-1$, the kernel 
function $k_{\lambda\beta}$ belongs to $\mathscr{C}^{\infty}(\mathbb{R}^{n}\setminus\{0\})$, is even, positive homogeneous of degree $-n$, 
has $\Delta^m\big[z\,k_{\lambda\beta}(z)\big]=0$ in $\mathbb{R}^{n}\setminus\{0\}$,}
\end{equation}
and satisfies the normalization condition 
\begin{equation}\label{eq:ugt-igFRD}
\int_{S^{n-1}}k_{\lambda\beta}(\xi)\,d{\mathcal{H}}^{n-1}(\xi)=\delta_{\lambda\beta}.
\end{equation}
In fact, at each $z\in\mathbb{R}^{n}\setminus\{0\}$ the kernel functions in \eqref{MSafddfgdsa-THM} are explicitly 
given by 
\begin{equation}\label{Kehsdbg-THM}
\qquad
k_{\lambda\beta}(z):=\frac{1}{\omega_{n-1}}
\sum_{\mu\leq\lambda\leq\mu+\beta}\frac{m}{|\mu|+1}\binom{\lambda}{\mu}\partial^{\mu}_z\bigg[\frac{1}{|z|^{n}}\bigg]\,
\frac{z^{\beta+\mu-\lambda}}{(\beta+\mu-\lambda)!}.
\end{equation}

\item\label{property4-INTRO} 
In a natural quantitative fashion, one has 
\begin{align}\label{a538bxd9-tk9-INTRO}
\mathcal{N}_\kappa(\nabla^{m}u)\in{\mathbb{X}}\Longleftrightarrow\dot{f}\in{\rm HWA}_{m-1}[{\mathbb{X}}_1].
\end{align}
\end{enumerate}

As a corollary, the ${\mathbb{X}}_1$-Regularity Problem for $\Delta^m$ in $\Omega$, formulated as
\begin{equation}\label{tk1-acxvtru-INTRO.bis.REG}
\begin{array}{l}
{\mathbb{X}}_1\text{-}{\rm RP}\,\,
\left\{
\begin{array}{l}
u\in{\mathscr{C}}^\infty(\Omega),\,\,\,\Delta^mu=0\,\text{ in }\,\Omega,
\\[8pt]
\mathcal{N}_\kappa(\nabla^{m-1} u)\in{\mathbb{X}}\,\,\text{ and }\,\,\mathcal{N}_\kappa(\nabla^{m}u)\in{\mathbb{X}},
\\[8pt]
(\partial^\gamma u)\big|^{{}^{\kappa\text{\rm -nt}}}_{\partial\Omega}=f_\gamma\,\,\text{ for each $\gamma\in{\mathbb{N}}_0^n$ with $|\gamma|\leq m-1$},
\\[2pt]
\text{for an arbitrary $\dot{f}=\{f_\gamma\}_{|\gamma|\leq m-1}\in{\rm HWA}_{m-1}[{\mathbb{X}}_1]$},
\end{array}
\right.
\end{array}
\end{equation}
is also well posed, in the sense of Hadamard{\rm :} a solution exists, is unique, and satisfies 
\begin{equation}\label{a538bxd9-tk2-INTRO.REG}
\|{\mathcal{N}}_\kappa(\nabla^{m-1}u)\|_{\mathbb{X}}
+\|{\mathcal{N}}_\kappa(\nabla^{m}u)\|_{\mathbb{X}}\leq C\|\dot{f}\|_{{\rm HWA}_{m-1}[{\mathbb{X}}_1]}
\end{equation}
for some $C\in(0,\infty)$ independent of $\dot{f}$. Furthermore, the unique solution of the 
${\mathbb{X}}_1$-Regularity Problem \eqref{tk1-acxvtru-INTRO.bis.REG} is representable as in \eqref{eq:U-INT-REP-FORM}
for some $\dot{g}=\{g_\gamma\}_{|\gamma|\leq m-1}\in{\rm HWA}_{m-1}[{\mathbb{X}}_1]$ which is uniquely determined 
{\rm (}and quantitatively controlled{\rm )} by the boundary datum $\dot{f}$.
\end{theorem}

Solving the ${\mathbb{X}}$-Dirichlet Problem \eqref{tk1-acxvtru-INTRO.bis} in the manner described in items {\it (i)}-{\it (iii)} of Theorem~\ref{DInRBVP-tk-INTRO} 
answers, for the polylaplacian, the question posed 70 years ago by S.~Agmon in \cite[p.\,180]{Ag}. Reflecting on his successful resolution of the Dirichlet 
problem for planar domains of class ${\mathscr{C}}^{1,r}$ with $r>1/2$ and scalar higher-order elliptic operators via multi-layer potential methods, 
Agmon points out that his approach relies on certain exceptional circumstances, namely the possibility of associating with the operator in question 
a multiple layer potential which is distinguished in the sense that the solution to the problem in the half-plane can be represented as a convolution 
on the boundary of the data with ``Poisson kernels'' of a special nature. He then goes on to note that ``{\it However, these results do not hold in 
general for higher order elliptic equations in more than two variables and this is the reason for the limitation of the method to the case of 
higher order equations in two variables only.}'' Agmon conjectured that, nevertheless, there should exist higher-order equations in more than 
two variables for which a potential theoretic approach based on integral operators of multi-layer type 
can be implemented successfully, and asked for a characterization of the class of such equations.

The work in this article demonstrates that the polyharmonic operator $\Delta^m$ in arbitrary Euclidean dimensions belongs to the aforementioned class. 
On the other hand, Mitrea-Mitrea-Mitrea proved in \cite{MMM.HOPT} that this is {\it not} the case for the homogeneous, constant real-matrix coefficient, 
symmetric, weakly elliptic, $n\times n$ system of order $2m$ in ${\mathbb{R}}^n$ given by 
\begin{equation}\label{eq:YGYFt-HFV}
\begin{array}{c}
I_{n\times n}\Delta^m-2\nabla\Delta^{m-1}{\rm div}=\Big(\delta_{jk}\Delta^m-2\Delta^{m-1}\partial_j\partial_k\Big)_{1\leq j,k\leq n}
\\[2pt]
\text{for each $n\in{\mathbb{N}}$ with $n\geq 2$ and each $m\in{\mathbb{N}}$}.\qquad
\end{array}
\end{equation}

An elegant feature of Theorem~\ref{DInRBVP-tk-INTRO} is that, in the particular case when $m=1$ and ${\mathbb{X}}$ belongs to the Lebesgue scale, 
it reduces precisely to the classical formulation of the Dirichlet problem for the Laplacian together with the familiar double layer potential 
representation of its solution. The latter framework has been developed by Fabes-Jodeit-Rivi\`ere in \cite{FJR}, Dahlberg-Kenig in \cite{DaKe}, 
and Verchota \cite{Ve1} for ${\mathscr{C}}^1$ and Lipschitz domains, by Hofmann-Mitrea-Taylor in \cite{HoMiTa10} for regular 
Semmes-Kenig-Toro domains, and by Mitrea-Mitrea-Mitrea in \cite{GHA.V} for $\delta$-{\rm AR} domains with a sufficiently small $\delta\in(0,1)$. 

What is however most remarkable is that, despite the final result exhibiting a form strikingly reminiscent of the classical Laplacian theory, 
neither the route leading to it nor the representation formula itself has a classical counterpart in the higher-order setting. Indeed, there is no a 
priori reason to expect that solutions of higher-order elliptic systems should admit a boundary layer representation of this nature, let alone one 
so closely paralleling the second-order theory. The identification of the correct potential-theoretic objects and the establishment of their mapping 
and invertibility properties require a substantially different analysis, reflecting genuinely new phenomena that arise at higher order. 
Consequently, while the above theorem may be viewed as a most faithful higher-order counterpart of the potential-theoretic treatment of the 
Dirichlet problem for the Laplacian via boundary layer methods, the integral representation formula and the manner in which this is employed 
in the treatment of higher-order Dirichlet and Regularity problems are genuinely new advances rather than extensions 
of previously known constructions.

Since the groundbreaking contributions of B.~Dahlberg, C.~Kenig, J.~Pipher, and G.~Verchota in the 1980s and 1990s, 
\cite{DKPV}, \cite{DKV1986}, \cite{PV1}, \cite{PV3}, \cite{PiVe1995}, \cite{Ver-p}, the theory of higher-order elliptic boundary value problems 
has remained largely confined to ${\mathscr{C}}^1$ and Lipschitz domains, with boundary data drawn from Lebesgue-type spaces. Building on the 
foundational work by D.~Mitrea, I.~Mitrea, M.~Mitrea in the series \cite{GHA.I}-\cite{GHA.V} and the recent monograph \cite{MMM.HOPT}, in this paper 
we establish well-posedness results involving classes of function spaces that are merely harmonic-analysis-friendly, and geometrically rough domains, 
tolerant of singularities not previously treated in the literature in connection with the higher-order elliptic boundary value problems considered here. 
For example, the latter class of domains permits the occurrence of certain boundary spiral points; see the discussion in \cite{GHA.V} and \cite{MMM.HOPT}. 
In particular, such domains need not be locally representable as upper graphs, placing them beyond the reach of techniques based on flattening the boundary. 
As a result, their geometric irregularity must be dealt with directly, rather than bypassed through changes of variables.

While all previous works on the Dirichlet problem for the polylaplacian with nontangential maximal function control of the solution are 
formulated exclusively in Lebesgue spaces, our higher-order Dirichlet problem is naturally posed in spaces of homogeneous Whitney arrays 
${\rm HWA}_{m-1}[{\mathbb X}]$ modeled on normed function spaces ${\mathbb X}$ belonging to the hierarchy of good spaces on $\partial\Omega$
(see \S\ref{Sec:HGS}). This introduces a novel functional-analytic dimension to the problem. In particular, any Generalized Banach Function Space ${\mathbb X}$ 
(cf. \cite[\S5.1]{GHA.II}) for which the Hardy-Littlewood maximal operator is bounded both on ${\mathbb X}$ and on its K\"othe dual ${\mathbb X}'$ 
is admissible. Such a framework already offers a plethora of examples, including Lebesgue spaces, sums and intersections of Lebesgue spaces, Muckenhoupt weighted 
Lebesgue spaces, Lorentz spaces, Morrey spaces, Herz spaces, and rearrangement-invariant Banach function spaces with Boyd indices in $(1,\infty)$.

\section{Background Material}
\label{Sec:2}

We begin by clarifying notation and terminology used in the formulation of Theorem~\ref{DInRBVP-tk-INTRO}.
We let ${\mathcal{L}}^n$ denote the Lebesgue measure 
in ${\mathbb{R}}^n$, and let ${\mathcal{H}}^{n-1}$ stand for the $(n-1)$-dimensional Hausdorff measure in ${\mathbb{R}}^n$. 
An {\it Ahlfors} {\it regular} {\it domain}, as defined in \cite[Definition~5.9.15, p.\,451]{GHA.I}, is an open set 
$\Omega\subseteq{\mathbb{R}}^n$ whose topological boundary $\partial\Omega$ is an Ahlfors regular set, in the sense of 
L.~Ahlfors \cite{Ahlfors1935} and G.~David \cite{David1984}, which satisfies 
\begin{align}\label{aakudgha.skdn.Cluj.Napoca}
&\hskip -0.10in
\mathcal{H}^{n-1}\Bigg(\Big\{x\in\partial\Omega:\,\limsup_{r\rightarrow 0^{+}}
\frac{\mathcal{L}^n\big(\Omega\cap B(x,r)\big)}{r^n}=0\text{ or} 
\\[0pt]
&\hskip 1.20in
\limsup_{r\rightarrow 0^{+}} 
\frac{\mathcal{L}^n\big((\mathbb{R}^n\setminus\Omega)\cap B(x,r)\big)}{r^n}=0\Big\}\Bigg)=0.
\nonumber
\end{align}
Condition \eqref{aakudgha.skdn.Cluj.Napoca} prevents $\partial\Omega$ from developing ``too many'' cusps, 
and also precludes $\Omega$ from having ``significant'' cracks. Any Ahlfors regular domain has an outward unit normal vector 
$\nu$, in the sense of De Giorgi-Federer (hence, geometric measure theoretic), which is well defined at $\mathcal{H}^{n-1}$-a.e. point 
on $\partial\Omega$ (see \cite[p.\,451]{GHA.I}). 

Given an Ahlfors regular domain $\Omega\subseteq\mathbb{R}^{n}$ and having picked an aperture parameter $\kappa>0$, 
define the nontangential approach regions 
\begin{equation}\label{NT-FF1}
\Gamma_{\kappa}(x):=\big\{y\in\Omega:\,|x-y|<(1+\kappa)\,{\rm dist}\,(y,\partial\Omega)\big\}\,\,\text{ for each }\,\,x\in\partial\Omega.
\end{equation}
For a Lebesgue measurable function $u:\Omega\to{\mathbb{C}}$, the {\tt nontangential} {\tt maximal} {\tt function} of $u$ 
with aperture $\kappa$ is defined as 
\begin{equation}\label{LDG-2Rd.NNN}
{\mathcal{N}}_{\kappa}u:\partial\Omega\longrightarrow[0,\infty],\quad
({\mathcal{N}}_{\kappa}u)(x):=\|u\|_{L^\infty(\Gamma_{\kappa}(x),{\mathcal{L}}^n)}\,\,\text{ for each }\,x\in\partial\Omega.
\end{equation}
Also, whenever $x\in\partial\Omega$ is such that $x\in\overline{\Gamma_{\kappa}(x)}$ and there exists a 
Lebesgue measurable set $N(x)\subset\Gamma_\kappa(x)$ satisfying ${\mathcal{L}}^n(N(x))=0$ for which the limit 
\begin{equation}\label{van-hTF}
\lim\limits_{(\Gamma_\kappa(x)\setminus N(x))\ni y\to x}u(y)\,\,\text{ exists in }\,\,{\mathbb{C}},
\end{equation}
its value is denoted by $\big(u\big|^{{}^{\kappa\text{\rm -nt}}}_{\partial\Omega}\big)(x)$ 
and is referred to as the {\tt nontangential} {\tt limit} of $u$ at $x$.

Continue to assume that  $\Omega\subseteq\mathbb{R}^{n}$ is an Ahlfors regular domain. Abbreviate $\sigma:={\mathcal{H}}^{n-1}\lfloor\partial\Omega$ 
and denote by $\nu=(\nu_1,\dots,\nu_n)$ the geometric measure-theoretic outward unit normal to $\Omega$. Following \cite{HoMiTa10}, for $j,k\in\{1,\dots,n\}$
the tangential derivative operator $\partial_{\tau_{jk}}$ acts on $\varphi\in\mathscr{C}^1$ near $\partial\Omega$ by
\begin{equation}\label{def-TAU-D}
\partial_{\tau_{jk}}\varphi:=\nu_j\big(\partial_k\varphi\big)\big|_{\partial\Omega}-\nu_k\big(\partial_j\varphi\big)\big|_{\partial\Omega}
\,\,\text{ at $\sigma$-a.e. point on }\,\,\partial\Omega,
\end{equation}
and extends, via the integration-by-parts identity $\int_{\partial\Omega}(\partial_{\tau_{jk}}\varphi)\psi\,d\sigma
=-\int_{\partial\Omega}\varphi(\partial_{\tau_{jk}}\psi)\,d\sigma$ (valid for $\psi\in\mathscr{C}^1_c(\mathbb{R}^n)$
by the De~Giorgi--Federer Divergence Theorem), to a distribution-valued operator on $L^1_{\rm loc}(\partial\Omega,\sigma)$.
As in \cite{GHA.II}, write $L^1_{1,{\rm loc}}(\partial\Omega,\sigma)$ for the local Sobolev space of order one on
$\partial\Omega$, consisting of $f\in L^1_{\rm loc}(\partial\Omega,\sigma)$ for which each $\partial_{\tau_{jk}}f$
is (canonically identified --\,cf. \cite[Corollary~3.7.3, p.\,283]{GHA.I}\,-- with) a function in $L^1_{\rm loc}(\partial\Omega,\sigma)$.

To introduce ``compatibility conditions'' of the sort discussed in the next definition, 
recall the elementary multi-indices $e_j:=(\delta_{jk})_{1\leq k\leq n}\in{\mathbb{N}}_0^n$, for $1\leq j\leq n$. 

\begin{definition}\label{def-CC}
Let $n,m\in\mathbb{N}$ with $n\geq 2$, and let $\Omega\subseteq\mathbb{R}^{n}$ be an Ahlfors regular domain. Set $\sigma:={\mathcal{H}}^{n-1}\lfloor\partial\Omega$, and let $\nu=(\nu_1,\dots,\nu_n)$ denote the geometric measure-theoretic outward unit normal to $\Omega$. 
Given a family of functions $\dot{f}=\{f_{\gamma}\}_{|\gamma|\leq m-1}$ with components in $L^1_{\rm loc}(\partial\Omega,\sigma)$, simply write 
$\dot{f}\in({\rm CC})_{m-1}$ to indicate that for every multi-index $\gamma\in\mathbb{N}^n_0$ with $|\gamma|\leq m-2$ and all $j,k\in\{1,\dots,n\}$ one has
\begin{equation}\label{CC}
f_\gamma \in L^1_{1,{\rm loc}}(\partial\Omega,\sigma)\,\,\text{ and }\,\,\partial_{\tau_{jk}}f_{\gamma}=\nu_jf_{\gamma+e_k}-\nu_kf_{\gamma+e_j}
\,\,\text{ at $\sigma$-a.e. point on }\partial\Omega.
\end{equation}

In the same setting, also define the space of Whitney arrays styled after $L_{\rm comp}^{1}$, the space of integrable functions with compact support, 
\begin{equation}\label{WAajasnfksjdnf}
{\rm WA}_{m-1}[L_{\rm comp}^{1}(\partial\Omega,\sigma)]:=\Big\{\dot{f}\in({\rm CC})_{m-1}:\,f_{\gamma}\in L_{\rm comp}^{1}(\partial\Omega,\sigma)
\,\,\text{ for each $|\gamma|\leq m-1$}\Big\}.
\end{equation}
\end{definition}

\begin{definition}\label{def:HWA-X}
Fix $n,m\in{\mathbb{N}}$ with $n\geq 2$. Let $\Omega\subseteq\mathbb{R}^{n}$ be an Ahlfors regular domain and 
abbreviate $\sigma:={\mathcal{H}}^{n-1}\lfloor\partial\Omega$. Also, let $\mathbb{X}\subseteq L^1_{\rm loc}(\partial\Omega,\sigma)$ be a normed vector space. 
In this setting, define the ${\mathbb{X}}$-{\tt based} {\tt homogeneous} {\tt Whitney} {\tt array} {\tt space}
\begin{align}\label{HWA.X}
{\rm HWA}_{m-1}\big[{\mathbb{X}}\big]:=\Big\{\dot{f}= & \{f_\gamma\}_{|\gamma|\leq m-1}\in({\rm CC})_{m-1}:\,
f_\gamma\in{\mathbb{X}}\,\,\text{ if }\,\,|\gamma|=m-1\,\text{ and}
\nonumber\\[0pt]
& f_\gamma\in L^1\Big(\partial\Omega,\frac{\sigma(x)}{1+|x|^{n+m-2-|\gamma|}}\Big)\,\,\text{ if $|\gamma|\leq m-2$}\Big\},
\end{align}
and equip it with the seminorm defined for each $\dot{f}\in{\rm HWA}_{m-1}\big[{\mathbb{X}}\big]$ by
\begin{equation}\label{eq:HWA-X-norm}
\|\dot{f}\|_{{\rm HWA}_{m-1}[{\mathbb{X}}]}:=\sum_{\gamma\in{\mathbb{N}}_0^n,\,|\gamma|=m-1}\|f_\gamma\|_{{\mathbb{X}}}.
\end{equation}

Setting, for each $\dot{f},\dot{g}\in{\rm HWA}_{m-1}\big[{\mathbb{X}}\big]$, 
\begin{equation}\label{eq:EQUIV-HWA}
\dot{f}\sim\dot{g}\,\overset{def}{\Longleftrightarrow}\,f_\gamma=g_\gamma\,\,\text{ whenever }\,\,|\gamma|=m-1
\end{equation}
defines an equivalence relation on ${\rm HWA}_{m-1}\big[{\mathbb{X}}\big]$.
For each $\dot{f}\in{\rm HWA}_{m-1}\big[{\mathbb{X}}\big]$ denote 
by $[\dot{f}]$ the equivalence class of $\dot{f}$, i.e., $[\dot{f}]:=\big\{\dot{g}\in{\rm HWA}_{m-1}\big[{\mathbb{X}}\big]:\,\dot{f}\sim\dot{g}\big\}$,
introduce the space of equivalence classes
\begin{equation}\label{eq:EQUIV-CLASS.DDD}
{\rm HWA}_{m-1}\big[{\mathbb{X}}\big]\big/\sim\,:=\Big\{[\dot{f}]:\,\dot{f}\in{\rm HWA}_{m-1}\big[{\mathbb{X}}\big]\Big\},
\end{equation}
and for each $\dot{f}=\{f_\gamma\}_{|\gamma|\leq m-1}\in{\rm HWA}_{m-1}\big[{\mathbb{X}}\big]$ {\rm (}unambiguously{\rm )} define
\begin{equation}\label{eq:EQUIV-CLASS.NNN}
\big\|[\dot{f}]\big\|_{{\rm HWA}_{m-1}[{\mathbb{X}}]/\sim}:=\sum_{\gamma\in{\mathbb{N}}_0^n,\,|\gamma|=m-1}\|f_\gamma\|_{{\mathbb{X}}}.
\end{equation}

Also consider ${\mathbb{X}}_1$-{\tt based} {\tt homogeneous} {\tt Whitney} {\tt array} {\tt space}
${\rm HWA}_{m-1}\big[{\mathbb{X}}_1\big]$ defined as above in relation to the Sobolev space
\begin{align}\label{INH.X1}
{\mathbb{X}}_1:=\Big\{f\in{\mathbb{X}}:\,\partial_{\tau_{jk}}f\in{\mathbb{X}}\,\text{ for }\,\,1\leq j,k\leq n\Big\},
\end{align}
equipped with the norm ${\mathbb{X}}_1\ni f\mapsto\|f\|_{{\mathbb{X}}_1}:=\|f\|_{\mathbb{X}}+\sum_{j,k=1}^n\|\partial_{\tau_{jk}}f\|_{\mathbb{X}}$.
\end{definition}

Given any $\ell\in{\mathbb{Z}}$, the symbol ${\mathcal{P}}_\ell$ is reserved for the space of all polynomials of degree $\leq\ell$ in $\mathbb{R}^n$ 
(with the understanding that ${\mathcal{P}}_\ell:=\{0\}$ if $\ell<0$). To formally define the Taylor polynomial functor (used in \eqref{eq:U-INT-REP-FORM}), 
fix a reference point $x_\ast\in{\mathbb{R}}^n$. For every smooth function $w$ defined in an open subset of ${\mathbb{R}}^n$ (containing $x_\ast$) then 
introduce\footnote{with the convention that a sum over an empty set of indices is interpreted as zero}
\begin{equation}\label{mxcvbprm3fdh.PP-INTRO}
{\mathbb{P}}_{\ell,x_\ast}[w](x):=
\sum_{|\eta|\leq\ell}\frac{1}{\eta!}(x-x_\ast)^{\eta}\big(\partial^{\eta}w\big)(x_\ast)
\,\text{ at each }\,x\in\mathbb{R}^{n}. 
\end{equation}
Thus, ${\mathbb{P}}_{\ell,x_\ast}[w]$ is simply the $\ell$-th order Taylor polynomial of $w$ (with base point $x_\ast$).
Consequently, ${\mathbb{P}}_{\ell,x_*}$ reproduces ${\mathcal{P}}_\ell$, i.e.,
\begin{equation}\label{eq:P-reproduces}
{\mathbb{P}}_{\ell,x_*}[P]=P\,\,\,\text{ for all }\,P\in{\mathcal{P}}_\ell.
\end{equation}
Also, Leibniz formula shows that for any two sufficiently 
smooth functions $w_1,w_2$ we have
\begin{align}\label{eq:leibniz-P}
{\mathbb{P}}_{\ell,x_*}[w_1w_2](x) 
=\sum_{|a|+|b|\leq\ell}\frac{(x-x_*)^a}{a!}(\partial^a w_1)(x_*)\frac{(x-x_*)^b}{b!}(\partial^b w_2)(x_*)\,\text{ for all }\,x\in{\mathbb{R}}^n.
\end{align}
Together with \eqref{eq:P-reproduces}, this shows that for every $P\in{\mathcal{P}}_\ell$, smooth $w$, and $x\in{\mathbb{R}}^n$ we have
\begin{equation}\label{eq:2bis}
{\mathbb{P}}_{\ell,x_*}[wP](x)=w(x_*)\,P(x)+\sum_{\substack{|a|+|b|\leq\ell\\ |a|>0}}
\frac{(x-x_*)^a}{a!}(\partial^a w)(x_*)\frac{(x-x_*)^b}{b!}(\partial^b P)(x_*).
\end{equation}

Retaining the geometric setting of Definition~\ref{def:HWA-X}, if for each $\ell\in{\mathbb{Z}}$ we define
\begin{equation}\label{eq:dot-P}
\dt{\mathcal{P}}_\ell:=\big\{\dot{P}=\big\{(\partial^\gamma P)\big|_{\partial\Omega}\big\}_{|\gamma|\leq m-1}:\,P\in{\mathcal{P}}_\ell\big\}
\end{equation}
then, as is apparent from definitions,
\begin{equation}\label{eq:incl}
\dt{\mathcal{P}}_{m-2}\subseteq{\rm HWA}_{m-1}\big[{\mathbb{X}}\big]\,\text{ and }\,\|\dot{f}\|_{{\rm HWA}_{m-1}[{\mathbb{X}}]}=0
\,\text{ for each }\,\dot{f}\in\dt{\mathcal{P}}_{m-2}.
\end{equation}

\section{Whitney arrays styled after the Hierarchy of Good Spaces}
\label{Sec:HGS}

We begin by recalling a number of definitions and results from \cite[\S5.1]{GHA.II} and \cite[\S9.1]{MiMiMiTa16}.

\begin{definition}\label{CIgsgC.Gad-DEF}
Let $({\mathscr{X}},\mathfrak{M},\mu)$ be a sigma-finite measure space and recall that $\mathscr{M}({\mathscr{X}},\mu)$ stands for 
the collection of all scalar-valued $\mu$-measurable functions on ${\mathscr{X}}$. A {\tt Generalized} {\tt Banach} {\tt Function} {\tt Space} 
{\rm (GBFS for short)} on ${\mathscr{X}}$ is a linear subspace ${\mathbb{X}}$ of ${\mathscr{M}}({\mathscr{X}},\mu)$ 
equipped with a norm $\|\cdot\|_{\mathbb{X}}$ satisfying the following three properties{\rm :}
\begin{enumerate}
\item[(1)] $[${\tt Order} {\tt Ideal}/{\tt Lattice} {\tt Property}$]$ If the functions $f\in{\mathbb{X}}$ and $g\in{\mathscr{M}}({\mathscr{X}},\mu)$ satisfy 
$|g|\leq|f|$ at $\mu$-a.e. point in ${\mathscr{X}}$ then $g\in{\mathbb{X}}$ and $\|g\|_{\mathbb{X}}\leq\|f\|_{\mathbb{X}}$.
\item[(2)] $[${\tt Fatou} {\tt Property}$]$ If the sequence $\{f_j\}_{j\in{\mathbb{N}}}\subseteq{\mathbb{X}}$ and the function 
$f\in{\mathscr{M}}({\mathscr{X}},\mu)$ satisfy $0\leq f_j\nearrow f$ pointwise $\mu$-a.e. on ${\mathscr{X}}$ as $j\to\infty$ and 
$\sup\limits_{j\in{\mathbb{N}}}\|f_j\|_{\mathbb{X}}<\infty$ it follows that $f\in{\mathbb{X}}$ and 
$\|f\|_{\mathbb{X}}=\sup\limits_{j\in{\mathbb{N}}}\|f_j\|_{\mathbb{X}}$.
\item[(3)] $[${\tt Richness} {\tt Property}$]$ There exists a sequence of sets $\{X_j\}_{j\in{\mathbb{N}}}\subseteq{\mathfrak{M}}$
with the property that $\bigcup_{j\in{\mathbb{N}}}X_j={\mathscr{X}}$ and ${\mathbf{1}}_{X_j}\in{\mathbb{X}}$ for every 
$j\in{\mathbb{N}}$.
\end{enumerate}
\end{definition}

If the condition imposed in item {\rm (3)} above is strengthened to the demand that ${\mathbf{1}}_{E}\in{\mathbb{X}}$ for every 
$E\in{\mathfrak{M}}$ with $\mu(E)<\infty$, then the space $\mathbb{X}$ is called a (classical) Banach function space on ${\mathscr{X}}$.
The latter concept appears in, e.g., \cite[Definition~1.1, p.\,2]{BeSha}, and is very popular in the mathematical literature. 
However, for the goals we have in mind, this is prohibitively restrictive since it excludes some basic function spaces, such as  
Muckenhoupt weighted Lebesgue spaces.

The associated space of a Generalized Banach Function Space ${\mathbb{X}}$ on ${\mathscr{X}}$, also known as K\"othe dual, is denoted by 
$\mathbb{X}'$ and consists of those measurable functions that pair ``well'' (in an integral sense on ${\mathscr{X}}$) with all functions 
in ${\mathbb{X}}$:
\begin{equation}\label{assoc_norm}
\mathbb{X}':=\Big\{g\in\mathscr{M}({\mathscr{X}},\mu):\,\int_{\mathscr{X}}|fg|\,d\mu<\infty\,\text{ for each }\,f\in\mathbb{X}\Big\}.
\end{equation}
If for each $g\in\mathbb{X}'$ one defines
\begin{equation}\label{assoc_norm.NORM}
\|g\|_{\mathbb{X}'}:=\sup\Big\{\int_{\mathscr{X}}|fg|\,d\mu:\,f\in\mathbb{X},\;\|f\|_{\mathbb{X}}\leq 1\Big\}, 
\end{equation}
then $\big(\mathbb{X}',\|\cdot\|_{\mathbb{X}'}\big)$ is itself a Generalized Banach Function Space on ${\mathscr{X}}$ (cf. \cite[\S5.1]{GHA.II}).

Any given Ahlfors regular set $\Sigma\subseteq{\mathbb{R}}^n$ becomes a measure space when equipped with the measure $\sigma:={\mathcal{H}}^{n-1}\lfloor\Sigma$. 
In such a setting, it also makes sense to consider the Hardy-Littlewood operator $\mathcal{M}_\Sigma$, acting on each $\sigma$-measurable function 
$f:\Sigma\to{\mathbb{C}}$ via
\begin{equation}\label{eq:HL_MaxOP}
(\mathcal{M}_\Sigma f)(x):=\sup_{r>0}\fint_{\Sigma\cap B(x,r)}|f|\,d\sigma\,\,\text{ at each point }\,\,x\in\Sigma.
\end{equation}
A Generalized Banach Function Space ${\mathbb{X}}$ on $\Sigma$ is called {\rm HA}-{\tt friendly} (harmonic-analysis-friendly) provided 
the Hardy-Littlewood operator $\mathcal{M}_\Sigma$ is bounded both on ${\mathbb{X}}$ and its K\"othe dual ${\mathbb{X}}'$.
As formalized in the definition below, which was first introduced in \cite{MMM.HOPT}, the class of {\rm HA}-friendly {\rm GBFS} is part of the foundational layer of an inductive construction that produces, through sums, intersections, completions, and interpolation, an 
infinite hierarchy of successive layers. The functional-analytic operations used to move from one level to the next are specifically 
chosen to interact naturally with the central objects of this framework, namely Calder\'on-Zygmund integral operators. Indeed, these 
operators are known to act boundedly on Muckenhoupt weighted Lebesgue spaces, and this stability extends, by means of extrapolation, 
to {\rm HA}-friendly {\rm GBFS} (see \cite[\S5.2]{GHA.II}). The inductive procedure then propagates these mapping properties throughout 
the entire hierarchy generated by the framework.

\begin{definition}\label{DEF.HGS}
Fix $n\in{\mathbb{N}}$ with $n\geq 2$. Let $\Sigma\subseteq\mathbb{R}^n$ be a closed Ahlfors regular set and abbreviate 
$\sigma:=\mathcal{H}^{n-1}\lfloor\Sigma$. Define the {\tt hierarchy} {\tt of} {\tt good} {\tt spaces} {\tt on} $\Sigma$ as the union 
\begin{equation}\label{eq:LEVELS}
\bigcup_{k=1}^\infty{\mathscr{L}}_k
\end{equation}
of {\tt levels} ${\mathscr{L}}_1,{\mathscr{L}}_2,\dots$ defined inductively as follows. First, ${\mathscr{L}}_1$ is taken to be the collection of 
Generalized Banach Function Spaces ${\mathbb{X}}$ on $\Sigma$ which are {\rm HA}-friendly and their ``cores,'' i.e., 
\begin{equation}\label{eq:TYRFFRFDGF8g56.iii}
\mathring{\mathbb{X}}:=\overline{L^\infty_{\rm comp}(\Sigma,\sigma)}^{\|\cdot\|_{\mathbb{X}}}.
\end{equation}
Next, assuming ${\mathscr{L}}_k$ has been already defined for some $k\in{\mathbb{N}}$, take ${\mathscr{L}}_{k+1}$ to be the collection 
\begin{align}\label{eq:LEVELS.2}
{\mathscr{L}}_{k+1}:=\Big\{{\mathbb{X}}+{\mathbb{Y}},\,\,\,{\mathbb{X}}\cap{\mathbb{Y}},\,\,\,
\overline{\mathbb{X}\cap\mathbb{Y}}{}^{\|\cdot\|_{\mathbb{X}}},\,\,\,({\mathbb{X}},{\mathbb{Y}})_{\theta,q}:\,
{\mathbb{X}},{\mathbb{Y}}\in{\mathscr{L}}_k,\,\,\,\theta\in(0,1),\,\,\,q\in[1,\infty]\Big\},
\end{align}
where ${\mathbb{X}}+{\mathbb{Y}}$ and ${\mathbb{X}}\cap{\mathbb{Y}}$ are equipped with the usual norm ``sum'' and ``intersection'' norms, 
the space $\overline{\mathbb{X}\cap\mathbb{Y}}{}^{\|\cdot\|_{\mathbb{X}}}$ is equipped with the norm inherited from ${\mathbb{X}}$, and 
$({\mathbb{X}},{\mathbb{Y}})_{\theta,q}$, the intermediate space produced by the real interpolation method, is equipped with the usual norm 
induced by Peetre's {\rm K}-functional {\rm (}cf., e.g., \cite{BeLo76}{\rm )}.
\end{definition}

We now record a key result concerning the functional analytic properties of the brand of spaces of 
homogeneous Whitney arrays introduced in Definition~\ref{def:HWA-X}. Its proof in \cite{MMM.HOPT} makes use
of the full force of the Calder\'on-Zygmund theory for multi-layer potential operators in {\rm NTA} domains with 
Ahlfors regular boundaries developed there. 

\begin{theorem}\label{thm:HWA-X}
Fix $n,m\in{\mathbb{N}}$ with $n\geq 2$. Let $\Omega\subseteq\mathbb{R}^{n}$ be a two-sided {\rm NTA} domain with an unbounded 
Ahlfors regular boundary, and abbreviate $\sigma:={\mathcal{H}}^{n-1}\lfloor\partial\Omega$. Also, let ${\mathbb{X}}$ be a normed 
space belonging to the hierarchy of good spaces on $\partial\Omega$. Then the following properties are true{\rm :}

\vskip 0.08in
(1) For each $\dot{f}\in{\rm HWA}_{m-1}\big[{\mathbb{X}}\big]$ one has 
$\|\dot{f}\|_{{\rm HWA}_{m-1}[{\mathbb{X}}]}=0$ if and only if $\dot{f}\in\dt{\mathcal{P}}_{m-2}$. 
In addition, for each $\dot{f},\dot{g}\in{\rm HWA}_{m-1}\big[{\mathbb{X}}\big]$ one has
$\dot{f}\sim\dot{g}$ if and only if $\dot{f}-\dot{g}\in\dt{\mathcal{P}}_{m-2}$.
In particular, ${\rm HWA}_{m-1}\big[{\mathbb{X}}\big]\big/\sim\,\,=\,{\rm HWA}_{m-1}\big[{\mathbb{X}}\big]\Big/\dt{\mathcal{P}}_{m-2}$.

\vskip 0.08in
(2) The quotient space $\big({\rm HWA}_{m-1}\big[{\mathbb{X}}\big]\big/\sim\,,\,\|\cdot\|_{{\rm HWA}_{m-1}[{\mathbb{X}}]/\sim}\big)$ 
is complete, hence Banach. Likewise, the quotient space 
$\big({\rm HWA}_{m-1}\big[{\mathbb{X}}_1\big]\big/\sim\,,\,\|\cdot\|_{{\rm HWA}_{m-1}[{\mathbb{X}}_1]/\sim}\big)$ 
is also Banach.
\end{theorem}

The theorem below, proved in \cite{MMM.HOPT}, contains boundary multi-trace results involving the spaces of homogeneous Whitney arrays 
from Definition~\ref{def:HWA-X}. This shows that the spaces of boundary data in the higher-order Dirichlet and Regularity problems in 
Theorem~\ref{DInRBVP-tk-INTRO} are optimal. 

\begin{theorem}\label{thm:HWA-X.aaa}
Fix integers $n,m\in\mathbb{N}$ with $n\geq 2$. Let $\Omega\subseteq\mathbb{R}^{n}$ be a two-sided {\rm NTA} domain whose boundary is 
unbounded and Ahlfors regular. Define $\sigma:=\mathcal{H}^{n-1}\lfloor\partial\Omega$ and pick an aperture parameter $\kappa\in(0,\infty)$. 
Also, let ${\mathbb{X}}$ be a normed space belonging to the hierarchy of good spaces on $\partial\Omega$. Finally, select a function 
$u\in\mathscr{C}^{\infty}(\Omega)$. Then the following properties are true.

(1) Under the assumption that  
\begin{equation}\label{jgjg3434.akshg}
\begin{array}{c}
\mathcal{N}_{\kappa}(\nabla^{m-1}u)\in\mathbb{X}\,\,\text{ and $(\partial^\gamma u)\big|_{\partial\Omega}^{{}^{\kappa\text{\rm -nt}}}$ 
exists $\sigma$-a.e. on $\partial\Omega$}
\\[4pt]
\text{for each multi-index $\gamma\in{\mathbb{N}}_0^n$ of length $|\gamma|=m-1$},
\end{array}
\end{equation}
one concludes that the nontangential trace $(\partial^\gamma u)\big|_{\partial\Omega}^{{}^{\kappa\text{\rm -nt}}}$ exists 
at $\sigma$-a.e. point on $\partial\Omega$ for each multi-index $\gamma\in{\mathbb{N}}_0^n$ of length $|\gamma|\leq m-1$, one has the membership
\begin{equation}\label{eq:TRACE.mmm}
{\rm Tr}^{{}^{\kappa\text{\rm -nt}}}_{m-1}u:=\Big\{(\partial^\gamma u)\big|_{\partial\Omega}^{{}^{\kappa\text{\rm -nt}}}\Big\}_{|\gamma|\leq m-1}
\in{\rm HWA}_{m-1}\big[{\mathbb{X}}\big],
\end{equation}
and there exists some constant $C=C(\Omega,{\mathbb{X}})\in(0,\infty)$ independent of $u$ for which 
\begin{equation}\label{eq:TRACE.mmm.EST}
\Big\|{\rm Tr}^{{}^{\kappa\text{\rm -nt}}}_{m-1}u\Big\|_{{\rm HWA}_{m-1}[{\mathbb{X}}]}
\leq C\big\|\mathcal{N}_{\kappa}(\nabla^{m-1}u)\big\|_{\mathbb{X}}.
\end{equation}

(2) If in place of \eqref{jgjg3434.akshg} one now assumes that $\mathcal{N}_{\kappa}(\nabla^{m-1}u)\in\mathbb{X}$ and 
$\mathcal{N}_{\kappa}(\nabla^{m}u)\in\mathbb{X}$, then the nontangential trace $(\partial^\gamma u)\big|_{\partial\Omega}^{{}^{\kappa\text{\rm -nt}}}$ 
exists at $\sigma$-a.e. point on $\partial\Omega$ for each multi-index $\gamma\in{\mathbb{N}}_0^n$ of length $|\gamma|\leq m-1$, one has the membership
${\rm Tr}^{{}^{\kappa\text{\rm -nt}}}_{m-1}u\in{\rm HWA}_{m-1}\big[{\mathbb{X}_{1}}\big]$,
and there exists some constant $C=C(\Omega,{\mathbb{X}})\in(0,\infty)$ independent of $u$ with the property that
\begin{equation}\label{eq:TRACE.mmm.EST-X1}
\Big\|{\rm Tr}^{{}^{\kappa\text{\rm -nt}}}_{m-1}u\Big\|_{{\rm HWA}_{m-1}[{\mathbb{X}}_1]}
\leq C\big\|\mathcal{N}_{\kappa}(\nabla^{m-1}u)\big\|_{\mathbb{X}}+C\big\|\mathcal{N}_{\kappa}(\nabla^{m}u)\big\|_{\mathbb{X}}.
\end{equation}
\end{theorem}

\section{The Multi-Layer Operator Associated with the Polylaplacian}
\label{Sec:Multilayer}

An elementary but useful observation is that if $n,m\in\mathbb{N}$ and $U$ is an open set in 
$\mathbb{R}^{n}$ then
\begin{equation}\label{HSjhagsdfjhgsgf}
\parbox{9.8cm}{for each harmonic function $u:U\to\mathbb{C}$, and each polynomial
$P\in\mathcal{P}_{m-1}$, the function $uP$ is a null-solution of $\Delta^{m}$ in $U$.}
\end{equation}
This follows by induction on $m\in\mathbb{N}$ relying on the identity 
$\Delta(uP)=2\nabla u \cdot \nabla P + u(\Delta P)$ in $U$.

\begin{definition}\label{Desidhgkjh} 
Fix $n,m\in\mathbb{N}$ with $n\geq 2$, and let $\omega_{n-1}$ denote the surface area of the unit sphere in $\mathbb{R}^{n}$. 
Let $\Omega\subseteq\mathbb{R}^{n}$ be an Ahlfors regular domain. Denote by $\nu$ its geometric measure theoretic outward unit normal
and abbreviate $\sigma:=\mathcal{H}^{n-1}\lfloor{\partial\Omega}$. In this setting, for each family $\dot{f}=\{f_{\gamma}\}_{|\gamma|\leq m-1}$ 
of complex-valued functions with the property that
\begin{equation}\label{Asdsdgsdgdfg}
f_\gamma\in L_{\rm comp}^{1}(\partial\Omega,\sigma)
\,\,\text{ for each $\gamma\in{\mathbb{N}}_0^n$ with $|\gamma|\leq m-1$}, 
\end{equation}
at each $x\in\mathbb{R}^{n}\setminus\partial\Omega$ define
\begin{equation}\label{Qlsdfpgjsdoigj}
(\dot{\mathcal{D}}_{\Delta^{m}}\dot{f})(x):=\frac{1}{\omega_{n-1}}\sum_{|\gamma|\leq m-1}
\int_{\partial\Omega}\frac{\nu(y)\cdot (y-x)}{|x-y|^{n}}\frac{(x-y)^{\gamma}}{\gamma!}f_{\gamma}(y)\,d\sigma(y).
\end{equation}
\end{definition}
We shall refer to $\dot{\mathcal{D}}_{\Delta^{m}}$ as the {\tt boundary}-{\tt to}-{\tt domain} 
{\tt multi}-{\tt layer} {\tt potential} {\tt operator} associated with the set $\Omega$ and the polylaplacian $\Delta^m$.
From \eqref{Qlsdfpgjsdoigj} and \eqref{HSjhagsdfjhgsgf} it follows that 
\begin{equation}\label{Multilayerprojsnf}
\begin{array}{c}
\dot{\mathcal{D}}_{\Delta^{m}}\dot{f}\,\,\text{ belongs to ${\mathscr{C}}^\infty(\mathbb{R}^{n}\setminus\partial\Omega)$, 
and $\Delta^{m}(\dot{\mathcal{D}}_{\Delta^{m}}\dot{f})=0$ in $\mathbb{R}^{n}\setminus\partial\Omega$, }
\\[2pt]
\text{for each family $\dot{f}=\{f_\gamma\}_{|\gamma|\leq m-1}$ of functions as in \eqref{Asdsdgsdgdfg}}.
\end{array}
\end{equation}
Corresponding to $m=1$ we have $\dot{\mathcal{D}}_{\Delta^1}={\mathcal{D}}_\Delta$, the classical boundary-to-domain 
harmonic double layer potential operator. Furthermore, the boundary-to-domain multi-layer potential operator associated with the 
polylaplacian $\Delta^{m}$ ``reproduces" polynomials of degree $\leq m-1$ in the precise manner detailed below. 

\begin{proposition}\label{MAahsifuhfiouh} 
Fix $n,m\in\mathbb{N}$, with $n\geq 2$. Let $\Omega\subseteq\mathbb{R}^{n}$ be an 
Ahlfors regular domain with compact boundary. Denote by $\nu$ its geometric measure-theoretic outward unit normal, and 
set $\sigma:=\mathcal{H}^{n-1}\lfloor{\partial\Omega}$. Then for each $P\in\mathcal{P}_{m-1}$ one has
\begin{equation}\label{Afhgfdsffgh}
\dot{\mathcal{D}}_{\Delta^{m}}\dot{P}=
\left\{
\begin{array}{@{}l@{\quad}l@{\qquad}l@{}}
P & \text{in } \Omega, & \text{if } \Omega \text{ is bounded},
\\[6pt]
0 & \text{in } \mathbb{R}^{n}\setminus\overline{\Omega}, & \text{if } \Omega \text{ is bounded},
\\[6pt]
0 & \text{in } \Omega, & \text{if } \Omega \text{ is unbounded},
\\[6pt]
-P & \text{in } \mathbb{R}^{n}\setminus\overline{\Omega}, & \text{if } \Omega \text{ is unbounded}.
\end{array}
\right.
\end{equation}
\end{proposition}

\begin{proof} 
Fix $P\in\mathcal{P}_{m-1}$ and $x\in\mathbb{R}^{n}\setminus\partial\Omega$. Then \eqref{Qlsdfpgjsdoigj} and Taylor's formula give
\begin{align*}
(\dot{\mathcal{D}}_{\Delta^{m}}\dot{P})(x)=\frac{1}{\omega_{n-1}}\int_{\partial\Omega}
\frac{\nu(y)\cdot (y-x)}{|x-y|^{n}}\Bigg(\sum_{|\gamma|\leq m-1}\frac{(x-y)^{\gamma}}{\gamma!}
(\partial^\gamma P)(y)\Bigg)d\sigma(y)
=P(x)(\mathcal{D}_{\Delta}1)(x).
\end{align*}
Now \eqref{Afhgfdsffgh} follows from this and Gauss' solid angle theorem \cite[Theorema Quartum, p.\,9]{Gauss}, asserting that 
$\mathcal{D}_{\Delta}1=1$ in $\Omega$ and $\mathcal{D}_{\Delta}1=0$ in $\mathbb{R}^{n}\setminus\overline{\Omega}$ if $\Omega$ is bounded, 
while $\mathcal{D}_{\Delta}1=0$ in $\Omega$ and $\mathcal{D}_{\Delta}1=-1$ in $\mathbb{R}^{n}\setminus\overline{\Omega}$ if $\Omega$ is 
unbounded (cf., e.g., \cite[Proposition~1.3.6, (1.3.46), pp.\,33--34]{GHA.IV}, presently used with $L:=\Delta$). 
\end{proof}

\begin{theorem}\label{Sfnsmdnf} 
Fix $n,m\in\mathbb{N}$, with $n\geq 2$. Let $\Omega\subseteq\mathbb{R}^{n}$ be an Ahlfors regular domain. 
Denote by $\nu$ its geometric measure-theoretic outward unit normal, and set $\sigma:=\mathcal{H}^{n-1}\lfloor{\partial\Omega}$. 

Then for each Whitney array $\dot{f}=\{f_{\gamma}\}_{|\gamma|\leq m-1}\in{\rm WA}_{m-1}[L_{\rm comp}^{1}(\partial\Omega,\sigma)]$, 
each multi-index $\lambda\in\mathbb{N}_{0}^{n}$ with $|\lambda|=m-1$, and each point $x\in\mathbb{R}^{n}\setminus\partial\Omega$ one has
\begin{align}\label{Eqwnwiefn}
\partial^{\lambda}(\dot{\mathcal{D}}_{\Delta^{m}}\dot{f})(x)&=\frac{1}{\omega_{n-1}}\sum_{\mu\leq\lambda}\sum_{|\gamma|=|\mu|}
\frac{m}{|\mu|+1}\binom{\lambda}{\mu}\int_{\partial\Omega} \nu(y)\cdot (y-x) \times
\nonumber\\
&\hspace{3cm}
\times\partial_{x}^{\mu}\bigg[\frac{1}{|x-y|^{n}}\bigg]\frac{(x-y)^{\gamma}}{\gamma!} f_{\gamma+\lambda-\mu}(y)\,d\sigma(y).
\end{align}
\end{theorem}

\begin{proof} 
Let $E_\Delta$ denote the canonical fundamental solution for the Laplacian (cf., e.g., \cite[Theorem~7.2, pp.\,235--236]{DM}).
Consider $\vec{k}=(k_{j})_{1\leq j\leq n}\in\big[{\mathscr{C}}^{\infty}(\mathbb{R}^{n}\setminus\{0\})\big]^{n}$ given by 
\begin{equation}\label{EW34dfghh}
\vec{k}(z):=\frac{1}{\omega_{n-1}}\frac{z}{|z|^{n}}=\nabla E_\Delta(z)\,\,\,\,\text{for every $z\in\mathbb{R}^{n}\setminus\{0\}$},    
\end{equation}
so for each $\dot{f}=\{f_{\gamma}\}_{|\gamma|\leq m-1}\in{\rm WA}_{m-1}[L_{\rm comp}^{1}(\partial\Omega,\sigma)]$ and 
$x\in\mathbb{R}^{n}\setminus\partial\Omega$, we may recast \eqref{Qlsdfpgjsdoigj} as 
\begin{equation}\label{JHSakglksjdgh_1}
(\dot{\mathcal{D}}_{\Delta^m}\dot{f})(x)=-\int_{\partial\Omega} \nu(y)\cdot\vec{k}(x-y)
\Bigg(\sum_{|\gamma|\leq m-1}\frac{(x-y)^{\gamma}}{\gamma!}f_{\gamma}(y)\Bigg)d\sigma(y).
\end{equation}

The proof of \eqref{Eqwnwiefn} proceeds by induction on $m\in\mathbb{N}$. Corresponding to $m=1$,
formula \eqref{Eqwnwiefn} becomes precisely the familiar expression for the boundary-to-domain harmonic double layer potential operator 
${\mathcal{D}}_\Delta$ associated with the set $\Omega$.

Next, suppose that \eqref{Eqwnwiefn} holds for some $m\in\mathbb{N}$ with the goal of establishing \eqref{Eqwnwiefn} for $m+1$. To this end, fix 
$\dot{f}=\{f_{\gamma}\}_{|\gamma|\leq m}\in{\rm WA}_{m}[L_{\rm comp}^{1}(\partial\Omega,\sigma)]$, a multi-index 
$\eta\in\mathbb{N}_{0}^{n}$ with $|\eta|=m$, and $x\in\mathbb{R}^{n}\setminus\partial\Omega$. We need to show that 
\begin{align}\label{Induchypsdf}
\partial^{\eta}(\dot{\mathcal{D}}_{\Delta^{m+1}}\dot{f})(x)&=\frac{1}{\omega_{n-1}}
\sum_{\mu\leq\eta}\sum_{|\gamma|=|\mu|}\frac{m+1}{|\mu|+1}\binom{\eta}{\mu}\int_{\partial\Omega} \nu(y)\cdot (y-x) \times 
\nonumber\\
&\hspace{3cm}\times\partial_{x}^{\mu}\bigg[\frac{1}{|x-y|^{n}}\bigg]\frac{(x-y)^{\gamma}}{\gamma!}f_{\gamma+\eta-\mu}(y)d\sigma(y).
\end{align}
Keeping in mind that $\sum_{\ell=1}^{n}\eta_{\ell}=m$ we may write
\begin{equation}\label{Jsjkahfjkhasdg}
\partial^{\eta}(\dot{\mathcal{D}}_{\Delta^{m+1}}\dot{f})(x)
=\sum_{\lambda+e_{\ell}=\eta}\frac{\eta_{\ell}}{m}\,\partial_{x}^{\lambda}\partial_{x_{\ell}}
(\dot{\mathcal{D}}_{\Delta^{m+1}}\dot{f})(x).
\end{equation}
We evaluate the right-hand side of \eqref{Jsjkahfjkhasdg} in three steps.

\vskip 0.1in
{\bf Step one:} {\tt Treat $\partial_{x_{\ell}}(\dot{\mathcal{D}}_{\Delta^{m+1}}\dot{f})(x)$ for $\ell\in\{1,\dots,n\}$.} 
It follows from the Leibniz product rule and \eqref{JHSakglksjdgh_1} (written with $m$ replaced by $m+1$) that
\begin{align}\label{Aqassdg}
\partial_{x_{\ell}}(\dot{\mathcal{D}}_{\Delta^{m+1}}\dot{f})(x)&=-\int_{\partial\Omega} \nu(y)\cdot\partial_{x_{\ell}}\vec{k}(x-y)\Bigg(\sum_{|\gamma|\leq m}\frac{(x-y)^{\gamma}}{\gamma!} f_{\gamma}(y) \Bigg)d\sigma(y)\nonumber\\[6pt]
&\qquad-\int_{\partial\Omega} \nu(y)\cdot\vec{k}(x-y)\,\partial_{x_{\ell}}\Bigg(\sum_{|\gamma|\leq m}\frac{(x-y)^{\gamma}}{\gamma!} f_{\gamma}(y) \Bigg)d\sigma(y).
\end{align}
Regarding the last line above, for each $r\in\{0,1,\dots,m\}$ at $\sigma$-a.e. $y\in\partial\Omega$ we may write
\begin{align}\label{Mskjalkfsjdg}
\partial_{x_{\ell}}\Bigg(\sum_{|\gamma|=r}\frac{(x-y)^{\gamma}}{\gamma!} 
f_{\gamma}(y) \Bigg)&=\sum_{\substack{|\gamma|=r\\ e_{\ell}
\leq\gamma}}\frac{(x-y)^{\gamma-e_{\ell}}}{(\gamma-e_{\ell})!}f_{\gamma}(y)
=\sum_{|\gamma|=r-1}\frac{(x-y)^{\gamma}}{\gamma!}f_{\gamma+e_{\ell}}(y).
\end{align}
Make use of this (summing up over $r$) to conclude that the last line of \eqref{Aqassdg} is equal to
\begin{align}\label{Kaposkaoskf}
-\int_{\partial\Omega} \nu(y)\cdot\vec{k}(x-y)\Bigg(\sum_{|\gamma|\leq m-1}\frac{(x-y)^{\gamma}}{\gamma!} 
f_{\gamma+e_{\ell}}(y)\Bigg)d\sigma(y).
\end{align}
Set $\dot{f}^{(e_{\ell})}:=\{f_{\gamma+e_{\ell}}\}_{|\gamma|\leq m-1}$. Since the components of
$\dot{f}$ are in $L_{\rm comp}^{1}(\partial\Omega,\sigma)$, so are those of $\dot{f}^{(e_{\ell})}$. 
These memberships combined with \eqref{JHSakglksjdgh_1} imply that 
\begin{equation}\label{Ghsjagjhgfdg}
\text{the last line in \eqref{Aqassdg} is equal to $(\dot{\mathcal{D}}_{\Delta^{m}}\dot{f}^{(e_{\ell})})(x)$.}
\end{equation}

To treat the first line on the right-hand side of \eqref{Aqassdg}, start by writing it as $A+B$, where
\begin{align}\label{Kasgjhgasdf_A}
A:=&-\int_{\partial\Omega}\nu(y)\cdot\partial_{x_{\ell}}\vec{k}(x-y)
\Bigg(\sum_{|\gamma|\leq m-1}\frac{(x-y)^{\gamma}}{\gamma!}f_{\gamma}(y) \Bigg)d\sigma(y),
\\[6pt]
B:=&-\int_{\partial\Omega}\nu(y)\cdot\partial_{x_{\ell}}\vec{k}(x-y)
\Bigg(\sum_{|\gamma|=m}\frac{(x-y)^{\gamma}}{\gamma!}f_{\gamma}(y)\Bigg)d\sigma(y).
\label{Kasgjhgasdf_B}
\end{align}
Observing that $\operatorname{div}\vec{k}=0$ in $\mathbb{R}^{n}\setminus\{0\}$, we may invoke \cite[(5.1.198), p.\,564]{GHA.IV} 
to rewrite \eqref{Kasgjhgasdf_A} as 
\begin{equation}\label{Yqwdsdfsdg}
A=-\sum_{j=1}^{n}\int_{\partial\Omega} k_{j}(x-y)\,\partial_{\tau_{j\ell}(y)}
\Bigg(\sum_{|\gamma|\leq m-1}\frac{(x-y)^{\gamma}}{\gamma!} f_{\gamma}(y)\Bigg)d\sigma(y).
\end{equation}
We next claim that at $\sigma$-a.e. point $y\in\partial\Omega$ we have 
\begin{equation}\label{MNnhsabdjhfjkh}
\partial_{\tau_{j\ell}(y)}\Bigg(\sum_{|\gamma|\leq m-1}\frac{(x-y)^{\gamma}}{\gamma!}f_{\gamma}(y)\Bigg)
=\sum_{|\gamma|=m-1}\frac{(x-y)^{\gamma}}{\gamma!}\partial_{\tau_{j\ell}(y)}\big[f_{\gamma}(y)\big].
\end{equation}
To justify this claim, first use \cite[Proposition~11.1.15, p.\,700]{GHA.II} and the compatibility conditions \eqref{CC} to write, 
for each $\gamma\in\mathbb{N}_{0}^{n}$ with $|\gamma|\leq m-1$ and $\sigma$-a.e. point $y\in\partial\Omega$, 
\begin{align}\label{Pmasfnshdfn}
\partial_{\tau_{j\ell}(y)}\Bigg(\frac{(x-y)^{\gamma}}{\gamma!}f_{\gamma}(y) \Bigg)
&=-\Bigg(\nu_{j}(y)\frac{(x-y)^{\gamma-e_{\ell}}}{(\gamma-e_{\ell})!}\mathbf{1}_{\{e_{\ell}
\leq\gamma\}}-\nu_{\ell}(y)\frac{(x-y)^{\gamma-e_{j}}}{(\gamma-e_{j})!}\mathbf{1}_{\{e_{j}\leq\gamma\}}\Bigg)f_{\gamma}(y)
\nonumber\\[6pt]
&\qquad+\frac{(x-y)^{\gamma}}{\gamma!}\partial_{\tau_{j\ell}(y)}\big[f_{\gamma}(y)\big].
\end{align}
Focusing on the terms in the first line on the right-hand side of \eqref{Pmasfnshdfn}, the compatibility 
conditions satisfied by the array $\dot{f}$ (cf. \eqref{CC}) allow us to write
\begin{align}\label{Qewwydgfuysdgf}
&-\sum_{|\gamma|\leq m-1}\Bigg(\nu_{j}(y)\frac{(x-y)^{\gamma-e_{\ell}}}{(\gamma-e_{\ell})!}\mathbf{1}_{\{e_{\ell}
\leq\gamma\}}-\nu_{\ell}(y)\frac{(x-y)^{\gamma-e_{j}}}{(\gamma-e_{j})!}\mathbf{1}_{\{e_{j}\leq\gamma\}}\Bigg)f_{\gamma}(y)
\nonumber\\[4pt]
&\quad=-\sum_{|\gamma|\leq m-2}\frac{(x-y)^{\gamma}}{\gamma!}
\Big(\nu_{j}(y)f_{\gamma+e_{\ell}}(y)-\nu_{\ell}(y)f_{\gamma+e_{j}}(y)\Big)
\nonumber\\[4pt]
&\quad=-\sum_{|\gamma|\leq m-2}\frac{(x-y)^{\gamma}}{\gamma!}\partial_{\tau_{j\ell}(y)}\big[f_{\gamma}(y)\big].
\end{align}
Summing up the identity in \eqref{Pmasfnshdfn} over all $\gamma\in\mathbb{N}_{0}^{n}$ with $|\gamma|\leq m-1$
and making use of \eqref{Qewwydgfuysdgf} yields \eqref{MNnhsabdjhfjkh}.
In turn, \eqref{MNnhsabdjhfjkh} may be used in \eqref{Yqwdsdfsdg} to obtain
\begin{equation}\label{SUIhaiushdgsh}
A=-\sum_{j=1}^{n}\int_{\partial\Omega} k_{j}(x-y)
\Bigg(\sum_{|\gamma|=m-1}\frac{(x-y)^{\gamma}}{\gamma!}\partial_{\tau_{j\ell}(y)}\big[f_{\gamma}(y)\big]\Bigg)d\sigma(y).
\end{equation}
Invoking once again the compatibility conditions \eqref{CC}, we further compute 
\begin{align}\label{NJKShgsiudhgisuhdghjjj}
A &=-\sum_{j=1}^{n}\int_{\partial\Omega} k_{j}(x-y)
\Bigg(\sum_{|\gamma|=m-1}\frac{(x-y)^{\gamma}}{\gamma!}
\big[\nu_{j}(y)f_{\gamma+e_{\ell}}(y)-\nu_{\ell}(y)f_{\gamma+e_{j}}(y)\big]\Bigg)d\sigma(y)
\nonumber\\[4pt]
&=-\int_{\partial\Omega}\nu(y)\cdot\vec{k}(x-y)
\Bigg(\sum_{|\gamma|=m-1}\frac{(x-y)^{\gamma}}{\gamma!}f_{\gamma+e_{\ell}}(y)\Bigg)d\sigma(y)\nonumber\\[4pt]
&\quad+\sum_{j=1}^{n}\int_{\partial\Omega}\nu_{\ell}(y)k_{j}(x-y)
\Bigg(\sum_{|\gamma|=m-1}\frac{(x-y)^{\gamma}}{\gamma!}f_{\gamma+e_{j}}(y)\Bigg)d\sigma(y).
\end{align}
For the first term on the right-most side of \eqref{NJKShgsiudhgisuhdghjjj}, employ \eqref{Mskjalkfsjdg} (with $r=m$) in reverse to conclude that 
\begin{align}\label{Tyuwegfjhgwe}
&-\int_{\partial\Omega}\nu(y)\cdot\vec{k}(x-y)
\Bigg(\sum_{|\gamma|=m-1}\frac{(x-y)^{\gamma}}{\gamma!}f_{\gamma+e_{\ell}}(y)\Bigg)d\sigma(y)
\nonumber\\[4pt]
&\quad=-\int_{\partial\Omega}\nu(y)\cdot\vec{k}(x-y)\,\partial_{x_{\ell}}
\Bigg(\sum_{|\gamma|=m}\frac{(x-y)^{\gamma}}{\gamma!}f_{\gamma}(y)\Bigg)d\sigma(y).
\end{align}
For the last line in \eqref{NJKShgsiudhgisuhdghjjj}, use that 
$k_{j}(z)=\frac{1}{\omega_{n-1}}\frac{z_{j}}{|z|^{n}}$ and \cite[Lemma~2.22, p.\,1270]{HLMMM} to write 
\begin{align}\label{swopeopjweg}
&\sum_{j=1}^{n}\int_{\partial\Omega}\nu_{\ell}(y)k_{j}(x-y)
\Bigg(\sum_{|\gamma|=m-1}\frac{(x-y)^{\gamma}}{\gamma!}f_{\gamma+e_{j}}(y)\Bigg)d\sigma(y)
\nonumber\\[4pt]
&\quad=\frac{1}{\omega_{n-1}}\int_{\partial\Omega}\frac{\nu_{\ell}(y)}{|x-y|^{n}}
\Bigg(\sum_{j=1}^{n}\sum_{|\gamma|=m-1}\frac{(x-y)^{\gamma+e_{j}}}{\gamma!}f_{\gamma+e_{j}}(y)\Bigg)d\sigma(y)
\nonumber\\[4pt]
&\quad=\frac{m}{\omega_{n-1}}\int_{\partial\Omega}\frac{\nu_{\ell}(y)}{|x-y|^{n}}
\Bigg(\sum_{|\gamma|=m}\frac{(x-y)^{\gamma}}{\gamma!}f_{\gamma}(y)\Bigg)d\sigma(y).
\end{align}
Altogether, \eqref{NJKShgsiudhgisuhdghjjj}, \eqref{Tyuwegfjhgwe}, and \eqref{swopeopjweg} imply
\begin{align}\label{Mewqrwefg}
A=&-\int_{\partial\Omega}\nu(y)\cdot\vec{k}(x-y)\,\partial_{x_{\ell}}
\Bigg(\sum_{|\gamma|=m}\frac{(x-y)^{\gamma}}{\gamma!}f_{\gamma}(y)\Bigg)d\sigma(y)
\nonumber\\[4pt]
&\qquad+\frac{m}{\omega_{n-1}}\int_{\partial\Omega}\frac{\nu_{\ell}(y)}{|x-y|^{n}}
\Bigg(\sum_{|\gamma|=m}\frac{(x-y)^{\gamma}}{\gamma!}f_{\gamma}(y)\Bigg)d\sigma(y).
\end{align}
Recalling \eqref{Kasgjhgasdf_B}, from \eqref{Mewqrwefg} and Leibniz's product rule we obtain that $A+B$,
which is equal to the first line on the right-hand side of \eqref{Aqassdg}, is in fact equal to 
\begin{align}\label{Pasdnjksdfhbj}
A+B=&-\partial_{x_{\ell}}\Bigg[\int_{\partial\Omega}\nu(y)\cdot\vec{k}(x-y)
\Bigg(\sum_{|\gamma|=m}\frac{(x-y)^{\gamma}}{\gamma!}f_{\gamma}(y)\Bigg)d\sigma(y)\Bigg]
\nonumber\\[4pt]
&\qquad+\frac{m}{\omega_{n-1}}\int_{\partial\Omega}\frac{\nu_{\ell}(y)}{|x-y|^{n}}
\Bigg(\sum_{|\gamma|=m}\frac{(x-y)^{\gamma}}{\gamma!}f_{\gamma}(y)\Bigg)d\sigma(y).
\end{align}
In concert, \eqref{Aqassdg}, \eqref{Ghsjagjhgfdg}, and \eqref{Pasdnjksdfhbj} imply that for each $\ell\in\{1,\dots,n\}$ we have
\begin{align}\label{JKSnakjfhkjhdg}
\partial_{x_{\ell}}(\dot{\mathcal{D}}_{\Delta^{m+1}}\dot{f})(x)
&=(\dot{\mathcal{D}}_{\Delta^{m}}\dot{f}^{(e_{\ell})})(x)-\partial_{x_{\ell}}\Bigg[\int_{\partial\Omega}\nu(y)\cdot\vec{k}(x-y)
\Bigg(\sum_{|\gamma|=m}\frac{(x-y)^{\gamma}}{\gamma!}f_{\gamma}(y)\Bigg)\,d\sigma(y)\Bigg]\nonumber\\[6pt]
&\qquad+\frac{m}{\omega_{n-1}}\int_{\partial\Omega}\frac{\nu_{\ell}(y)}{|x-y|^{n}}
\Bigg(\sum_{|\gamma|=m}\frac{(x-y)^{\gamma}}{\gamma!}f_{\gamma}(y)\Bigg)d\sigma(y).
\end{align}

\vskip 0.1in
{\bf Step two:} {\tt Computing $\partial_{x}^{\lambda}\partial_{x_{\ell}}(\dot{\mathcal{D}}_{\Delta^{m+1}}\dot{f})(x)$ for 
$\lambda\in\mathbb{N}_{0}^{n}$ and $\ell\in\operatorname{supp}\eta$ with $\lambda+e_{\ell}=\eta$.} Note that these conditions 
imply $|\lambda|=m-1$. Applying $\partial_{x}^{\lambda}$ to both sides of \eqref{JKSnakjfhkjhdg} gives 
\begin{align}\label{NJKAshakjfhkshdg}
&\partial_{x}^{\lambda}\partial_{x_{\ell}}(\dot{\mathcal{D}}_{\Delta^{m+1}}\dot{f})(x)
\nonumber\\
&\quad=\partial^{\lambda}(\dot{\mathcal{D}}_{\Delta^{m}}\dot{f}^{(e_{\ell})})(x)-\partial_{x}^{\lambda}\partial_{x_{\ell}}
\Bigg[\int_{\partial\Omega}\nu(y)\cdot\vec{k}(x-y)
\Bigg(\sum_{|\gamma|=m}\frac{(x-y)^{\gamma}}{\gamma!}f_{\gamma}(y)\Bigg)d\sigma(y)\Bigg]
\nonumber\\[6pt]
&\qquad+\frac{m}{\omega_{n-1}}\partial_{x}^{\lambda}\Bigg[\int_{\partial\Omega}\frac{\nu_{\ell}(y)}{|x-y|^{n}}
\Bigg(\sum_{|\gamma|=m}\frac{(x-y)^{\gamma}}{\gamma!}f_{\gamma}(y)\Bigg)d\sigma(y)\Bigg].
\end{align}
The idea now is to invoke the induction hypothesis for $\partial^{\lambda}(\dot{\mathcal{D}}_{\Delta^{m}}\dot{f}^{(e_{\ell})})$.
As a prerequisite, we need to check that $\dot{f}^{(e_{\ell})}\in{\rm WA}_{m-1}[L_{\rm comp}^{1}(\partial\Omega,\sigma)]$, which presently 
comes down to verifying that $\dot{f}^{(e_{\ell})}\in({\rm CC})_{m-1}$. The latter membership follows by 
observing that under the current assumptions for each $\gamma\in\mathbb{N}_{0}^{n}$ with $|\gamma|\leq m-2$, 
and each $j,k\in\{1,\dots,n\}$, we have
\begin{align}\label{Qdmdskvmdsfg}
\partial_{\tau_{jk}}f_{\gamma}^{(e_{\ell})}=\partial_{\tau_{jk}}f_{\gamma+e_\ell}
=\nu_{j}f_{\gamma+e_{\ell}+e_{k}}-\nu_{k}f_{\gamma+e_{\ell}+e_{j}}
=\nu_{j}f_{\gamma+e_{k}}^{(e_{\ell})}-\nu_{k}f_{\gamma+e_{j}}^{(e_{\ell})}.
\end{align}
Granted this, the induction hypothesis for the current multi-index $\lambda$ allows us to write
\begin{align}\label{Wqposdkfposdkf-A}
\partial^{\lambda}\big(\dot{\mathcal{D}}_{\Delta^{m}}\dot{f}^{(e_{\ell})}\big)(x)
&=\frac{1}{\omega_{n-1}}\sum_{\mu\leq\lambda}\sum_{|\gamma|=|\mu|}
\frac{m}{|\mu|+1}\binom{\lambda}{\mu}\int_{\partial\Omega}\nu(y)\cdot (y-x)\times 
\nonumber\\[4pt]
&\hspace{3cm}\times\partial_{x}^{\mu}\bigg[\frac{1}{|x-y|^{n}}\bigg]\frac{(x-y)^{\gamma}}{\gamma!}\, 
f^{(e_{\ell})}_{\gamma+\lambda-\mu}(y)\,d\sigma(y).
\end{align}
Recalling that $\eta=\lambda+e_{\ell}$, hence $f^{(e_{\ell})}_{\gamma+\lambda-\mu}=f_{\gamma+\lambda-\mu+e_{\ell}}=f_{\gamma+\eta-\mu}$, we therefore obtain
\begin{align}\label{Wqposdkfposdkf-B}
\partial^{\lambda}\big(\dot{\mathcal{D}}_{\Delta^{m}}\dot{f}^{(e_{\ell})}\big)(x)
&=\frac{1}{\omega_{n-1}}\sum_{\mu\leq\eta-e_{\ell}}\sum_{|\gamma|=|\mu|}\frac{m}{|\mu|+1}\binom{\eta-e_{\ell}}{\mu}
\int_{\partial\Omega}\nu(y)\cdot (y-x)\times 
\nonumber\\[4pt]
&\hspace{3cm}\times\partial_{x}^{\mu}\bigg[\frac{1}{|x-y|^{n}}\bigg]\frac{(x-y)^{\gamma}}{\gamma!}\, 
f_{\gamma+\eta-\mu}(y)\,d\sigma(y).
\end{align}
Returning with \eqref{Wqposdkfposdkf-B} to \eqref{NJKAshakjfhkshdg} and also bearing in mind that 
$\partial_{x}^{\lambda}\partial_{x_{\ell}}=\partial^{\eta}_x$, we conclude that 
\begin{align}\label{Ttyudfghvbnm}
\partial_{x}^{\lambda}\partial_{x_{\ell}}(\dot{\mathcal{D}}_{\Delta^{m+1}}\dot{f})(x)
&=\frac{1}{\omega_{n-1}}\sum_{\mu\leq\eta-e_{\ell}}\sum_{|\gamma|=|\mu|}\frac{m}{|\mu|+1}\binom{\eta-e_{\ell}}{\mu}
\int_{\partial\Omega} \nu(y)\cdot(y-x)\times
\nonumber\\[2pt]
&\hspace{3cm}\times\partial_{x}^{\mu}\bigg[\frac{1}{|x-y|^{n}}\bigg]\frac{(x-y)^{\gamma}}{\gamma!}\, 
f_{\gamma+\eta-\mu}(y)\,d\sigma(y)
\nonumber\\[6pt]
&\qquad-\partial_{x}^{\eta}\Bigg[\int_{\partial\Omega}\nu(y)\cdot\vec{k}(x-y)\sum_{|\gamma|=m}\frac{(x-y)^{\gamma}}{\gamma!}
f_{\gamma}(y)d\sigma(y)\Bigg]
\nonumber\\[6pt]
&\qquad+\frac{m}{\omega_{n-1}}\,\partial_{x}^{\lambda}\Bigg[\int_{\partial\Omega}\frac{\nu_{\ell}(y)}{|x-y|^{n}}
\Bigg(\sum_{|\gamma|=m}\frac{(x-y)^{\gamma}}{\gamma!}f_{\gamma}(y)\Bigg)\,d\sigma(y)\Bigg].
\end{align}

\vskip 0.1in
{\bf Step three:} {\tt Computing $\partial^{\eta}(\dot{\mathcal{D}}_{\Delta^{m+1}}\dot{f})(x)$.} The idea is to combine
\eqref{Jsjkahfjkhasdg} and \eqref{Ttyudfghvbnm}: multiplying both sides of \eqref{Ttyudfghvbnm} by 
$\displaystyle\frac{\eta_{\ell}}{m}$ and summing over all decompositions $\lambda+e_{\ell}=\eta$ yields
\begin{equation}\label{JJJsknsdlhsd}
\partial^{\eta}(\dot{\mathcal{D}}_{\Delta^{m+1}}\dot{f})(x)=I_1+I_2+I_3
\end{equation}
where
\begin{align}\label{Afdsdgdfhdh}
I_1:=&\frac{1}{\omega_{n-1}}\sum_{\lambda+e_{\ell}=\eta}\frac{\eta_{\ell}}{m}\sum_{\mu\leq\eta-e_{\ell}}
\sum_{|\gamma|=|\mu|}\frac{m}{|\mu|+1}\binom{\eta-e_{\ell}}{\mu}\int_{\partial\Omega} \nu(y)\cdot (y-x)\times 
\nonumber\\[4pt]
&\hspace{4cm}\times\partial_{x}^{\mu}\bigg[\frac{1}{|x-y|^{n}}\bigg]\frac{(x-y)^{\gamma}}{\gamma!}\, 
f_{\gamma+\eta-\mu}(y)\,d\sigma(y),
\end{align}
\begin{align}\label{Afdsdgdfhdh_2}
I_2:=-\sum_{\lambda+e_{\ell}=\eta}\frac{\eta_{\ell}}{m}\,
\partial_{x}^{\eta}\Bigg[\int_{\partial\Omega}\nu(y)\cdot\vec{k}(x-y)\sum_{|\gamma|=m}\frac{(x-y)^{\gamma}}{\gamma!}
f_{\gamma}(y)\,d\sigma(y)\Bigg],
\end{align}
and 
\begin{align}\label{Afdsdgdfhdh_3}
I_3:=\frac{m}{\omega_{n-1}}\sum_{\lambda+e_{\ell}=\eta}\frac{\eta_{\ell}}{m}\,
\partial_{x}^{\lambda}\Bigg[\int_{\partial\Omega}\frac{\nu_{\ell}(y)}{|x-y|^{n}}
\Bigg(\sum_{|\gamma|=m}\frac{(x-y)^{\gamma}}{\gamma!}f_{\gamma}(y)\Bigg)\,d\sigma(y)\Bigg].
\end{align}

First focus on $I_1$. The summand in this term is independent of $\lambda\in\mathbb{N}_{0}^{n}$. Moreover, note that 
$\ell\in\operatorname{supp}\eta$ and for each $\mu\leq\eta-e_{\ell}$ we have 
$\eta_{\ell}\binom{\eta-e_{\ell}}{\mu}=(\eta_{\ell}-\mu_{\ell})\binom{\eta}{\mu}$. In light of this identity
and the fact that $\eta_{\ell}-\mu_{\ell}=0$ if $\mu=\eta$, we may extend the range of $\mu$ in the second sum in 
\eqref{Afdsdgdfhdh} to $\mu\leq\eta$. As a consequence, the expression in \eqref{Afdsdgdfhdh} becomes
\begin{align}\label{LPskashfskjdhglksjg}
I_1=&\frac{1}{\omega_{n-1}}\sum_{\ell\in\operatorname{supp}\eta}\sum_{\mu\leq\eta}
\sum_{|\gamma|=|\mu|}\frac{(\eta_{\ell}-\mu_{\ell})}{|\mu|+1}\binom{\eta}{\mu}
\int_{\partial\Omega}\nu(y)\cdot (y-x)\times 
\nonumber\\[4pt]
&\hspace{3cm}\times\partial_{x}^{\mu}\bigg[\frac{1}{|x-y|^{n}}\bigg]\frac{(x-y)^{\gamma}}{\gamma!}\, 
f_{\gamma+\eta-\mu}(y)\,d\sigma(y).
\end{align}
Upon noting that $\sum_{\ell\in\operatorname{supp}\eta}(\eta_{\ell}-\mu_{\ell})=m-|\mu|$ for $\mu\leq\eta$ further leads to  
\begin{align}\label{AMmsadklajsklfj}
I_1=&\frac{1}{\omega_{n-1}}\sum_{\mu\leq\eta}\sum_{|\gamma|=|\mu|}\frac{m-|\mu|}{|\mu|+1}\binom{\eta}{\mu}
\int_{\partial\Omega} \nu(y)\cdot (y-x) \times 
\nonumber\\[4pt]
&\hspace{3cm}\times\partial_{x}^{\mu}\bigg[\frac{1}{|x-y|^{n}}\bigg]\frac{(x-y)^{\gamma}}{\gamma!}\, 
f_{\gamma+\eta-\mu}(y)\,d\sigma(y).
\end{align}

In relation to $I_2$ from \eqref{Afdsdgdfhdh_2}, first note that $\sum_{\lambda+e_{\ell}=\eta}\eta_{\ell}=m$. Also recalling 
\eqref{EW34dfghh}, differentiating under the integral sign, and applying the Leibniz product rule produces 
\begin{align}\label{SKLjkljgj}
I_2 &=\frac{1}{\omega_{n-1}}\sum_{j=1}^{n}\sum_{\lambda\leq\eta}\binom{\eta}{\lambda}\int_{\partial\Omega}\nu_{j}(y)
\partial_{x}^{\eta-\lambda}\big[(y_{j}-x_{j})\big]\times
\nonumber\\[4pt]
&\hspace{3.8cm}
\times\partial_{x}^{\lambda}
\Bigg(\frac{1}{|x-y|^{n}}\sum_{|\gamma|=m}\frac{(x-y)^{\gamma}}{\gamma!}f_{\gamma}(y)\Bigg)\,d\sigma(y).
\end{align}
The presence of $\partial_{x}^{\eta-\lambda}\big[(y_{j}-x_{j})\big]$ forces $\eta-\lambda$ to be either $0$ or $e_j$, so 
\eqref{SKLjkljgj} further simplifies as
\begin{align}\label{Mjsaajkshfjhg}
I_2=&\frac{1}{\omega_{n-1}}\int_{\partial\Omega}\nu(y)\cdot (y-x)\, \partial_{x}^{\eta}
\Bigg(\frac{1}{|x-y|^{n}}\sum_{|\gamma|=m}\frac{(x-y)^{\gamma}}{\gamma!}f_{\gamma}(y)\Bigg)\,d\sigma(y)
\nonumber\\[4pt]
&\quad-\frac{1}{\omega_{n-1}}\sum_{\lambda+e_{j}=\eta}\eta_{j}\int_{\partial\Omega}\nu_{j}(y)\,\partial_{x}^{\lambda}
\Bigg(\frac{1}{|x-y|^{n}}\sum_{|\gamma|=m}\frac{(x-y)^{\gamma}}{\gamma!}f_{\gamma}(y)\Bigg)\,d\sigma(y).
\end{align}

Finally, upon differentiating under the integral sign in \eqref{Afdsdgdfhdh_3} we obtain 
\begin{align}\label{Frdghsfdg}
I_3=\frac{1}{\omega_{n-1}}\sum_{\lambda+e_{\ell}=\eta}\eta_{\ell}\int_{\partial\Omega}\nu_{\ell}(y)\,\partial_{x}^{\lambda}
\Bigg(\frac{1}{|x-y|^{n}}\sum_{|\gamma|=m}\frac{(x-y)^{\gamma}}{\gamma!}f_{\gamma}(y)\Bigg)\,d\sigma(y),
\end{align}
which is the opposite of the second line in \eqref{Mjsaajkshfjhg}. Bearing this in mind, from 
\eqref{Afdsdgdfhdh}, \eqref{AMmsadklajsklfj}, \eqref{Mjsaajkshfjhg}, and \eqref{Frdghsfdg} we conclude that 
\begin{align}\label{Yygfgfgbf}
\partial^{\eta}(\dot{\mathcal{D}}_{\Delta^{m+1}}\dot{f})(x)
&=\frac{1}{\omega_{n-1}}\sum_{\mu\leq\eta}\sum_{|\gamma|=|\mu|}\frac{m-|\mu|}{|\mu|+1}\binom{\eta}{\mu}
\int_{\partial\Omega} \nu(y)\cdot (y-x)\times 
\nonumber\\[4pt]
&\hspace{3cm}\times\partial_{x}^{\mu}\bigg[\frac{1}{|x-y|^{n}}\bigg]\frac{(x-y)^{\gamma}}{\gamma!}\,f_{\gamma+\eta-\mu}(y)\,d\sigma(y)
\nonumber\\[6pt]
&+\frac{1}{\omega_{n-1}}\int_{\partial\Omega}\nu(y)\cdot (y-x)\,\partial_{x}^{\eta}
\Bigg(\frac{1}{|x-y|^{n}}\sum_{|\gamma|=m}\frac{(x-y)^{\gamma}}{\gamma!}f_{\gamma}(y)\Bigg)\,d\sigma(y).
\end{align}
By the Leibniz product rule, the last line on the right-hand side of \eqref{Yygfgfgbf} may be expressed as 
\begin{align}\label{Qhdjkslasgh}
&\frac{1}{\omega_{n-1}}\int_{\partial\Omega}\nu(y)\cdot (y-x)\,\partial_{x}^{\eta}\Bigg(\frac{1}{|x-y|^{n}}
\sum_{|\gamma|=m}\frac{(x-y)^{\gamma}}{\gamma!}f_{\gamma}(y)\Bigg)d\sigma(y)\nonumber\\[4pt]
&\qquad\qquad
=\frac{1}{\omega_{n-1}}\sum_{\mu\leq\eta}\binom{\eta}{\mu}\int_{\partial\Omega} \nu(y)\cdot (y-x) \times
\nonumber\\[4pt]
&\hspace{3cm}\times\partial_{x}^{\mu}\bigg[\frac{1}{|x-y|^{n}}\bigg]\,\partial_{x}^{\eta-\mu}
\Bigg(\sum_{|\gamma|=m}\frac{(x-y)^{\gamma}}{\gamma!}\,f_{\gamma}(y)\Bigg)d\sigma(y).
\end{align}
In addition, iterating \eqref{Mskjalkfsjdg} gives that for each $\mu\leq\eta$ and at $\sigma$-a.e. point $y\in\partial\Omega$ we have
\begin{align*}
\partial_{x}^{\eta-\mu}\Bigg(\sum_{|\gamma|=m}\frac{(x-y)^{\gamma}}{\gamma!}\, f_{\gamma}(y)\Bigg)
=\sum_{|\gamma|=m-|\eta-\mu|}\frac{(x-y)^{\gamma}}{\gamma!}\, f_{\gamma+\eta-\mu}(y)
=\sum_{|\gamma|=|\mu|}\frac{(x-y)^{\gamma}}{\gamma!}\, f_{\gamma+\eta-\mu}(y),
\end{align*}
where we have used that $|\eta-\mu|=m-|\mu|$. Thus, the right-hand side of \eqref{Qhdjkslasgh}, 
which corresponds to the last line on the right-hand side of \eqref{Yygfgfgbf}, is equal to
\begin{align}\label{SAfdsdgsfdgdfsgh}
\frac{1}{\omega_{n-1}}\sum_{\mu\leq\eta}\sum_{|\gamma|=|\mu|}\binom{\eta}{\mu}\int_{\partial\Omega} \nu(y)\cdot (y-x)
\partial_{x}^{\mu}\bigg[\frac{1}{|x-y|^{n}}\bigg]\frac{(x-y)^{\gamma}}{\gamma!}\,f_{\gamma+\eta-\mu}(y)\,d\sigma(y).
\end{align}
Now \eqref{Induchypsdf} follows by adding the first line on the right-hand side of \eqref{Yygfgfgbf} to the expression 
in \eqref{SAfdsdgsfdgdfsgh}. This finishes the proof of Theorem~\ref{Sfnsmdnf}.
\end{proof}

\begin{theorem}\label{Thjdsbgkjbsdgkbh}
Fix $m,n\in\mathbb{N}$ with $n\geq 2$ and let $\Omega\subseteq\mathbb{R}^{n}$ be an Ahlfors regular domain.
Denote by $\nu$ its geometric measure-theoretic outward unit normal, and set $\sigma:=\mathcal{H}^{n-1}\lfloor{\partial\Omega}$. 
Then the family $\big\{k_{\lambda\beta}\big\}_{|\beta|=|\lambda|=m-1}$ from \eqref{Kehsdbg-THM} has the properties listed in 
\eqref{HShagdfsdg-THM}-\eqref{eq:ugt-igFRD}, and for each Whitney array $\dot{f}\in{\rm WA}_{m-1}[L_{\rm comp}^{1}(\partial\Omega,\sigma)]$ and 
each multi-index $\lambda\in\mathbb{N}_{0}^{n}$ with $|\lambda|=m-1$ one has
\begin{equation}\label{Thm-MSafddfgdsa}
\partial^{\lambda}\big(\dot{\mathcal{D}}_{\Delta^{m}}\dot{f}\big)(x)=\sum_{|\beta|=m-1}\int_{\partial\Omega}
\nu(y)\cdot (y-x)\, k_{\lambda\beta}(x-y)f_{\beta}(y)\,d\sigma(y)\,\text{ for all $x\in\mathbb{R}^{n}\setminus\partial\Omega$.}
\end{equation}
\end{theorem}

\begin{proof} 
Fix $\lambda\in\mathbb{N}_{0}^{n}$ with $|\lambda|=m-1$ and observe that, for each $\mu\in\mathbb{N}_{0}^{n}$
such that $\mu\leq\lambda$, the mapping $\gamma\mapsto\lambda-\mu+\gamma$ from  
$\{\gamma\in\mathbb{N}_{0}^{n}:|\gamma|=|\mu|\}$ into $\{\beta\in\mathbb{N}_{0}^{n}:\,|\beta|=m-1$ and 
$\lambda\leq\mu+\beta\}$ is a bijection. As a consequence, upon introducing the multi-index
\begin{equation}\label{Nsjhafjhdsga}
\beta:=\gamma+\lambda-\mu\,\,\,\,\text{for each $\gamma\in\mathbb{N}_{0}^{n}$ with $|\gamma|=|\mu|$},
\end{equation}
we may rewrite \eqref{Eqwnwiefn} as 
\begin{align}\label{Safkdsfgsdfgsdf}
\partial^{\lambda}(\dot{\mathcal{D}}_{\Delta^{m}}\dot{f})(x)
&=\frac{1}{\omega_{n-1}}\sum_{|\beta|=m-1}\,\sum_{\mu\leq\lambda\leq\mu+\beta}\frac{m}{|\mu|+1}\binom{\lambda}{\mu}
\int_{\partial\Omega} \nu(y)\cdot (y-x)\times 
\nonumber\\
&\hspace{2cm}\times\partial_{x}^{\mu}\bigg[\frac{1}{|x-y|^{n}}\bigg]
\frac{(x-y)^{\beta+\mu-\lambda}}{(\beta+\mu-\lambda)!}\,f_{\beta}(y)\,d\sigma(y)
\end{align}
for each $x\in\mathbb{R}^{n}\setminus\partial\Omega$ and each 
$\dot{f}=\{f_{\gamma}\}_{|\gamma|\leq m-1}\in{\rm WA}_{m-1}[L_{\rm comp}^{1}(\partial\Omega,\sigma)]$.
Now \eqref{Thm-MSafddfgdsa} follows from \eqref{Safkdsfgsdfgsdf} upon recalling \eqref{Kehsdbg-THM}. 

To show that the properties in \eqref{HShagdfsdg-THM} are satisfied, fix $\lambda,\beta\in\mathbb{N}_{0}^{n}$ 
with $|\lambda|=|\beta|=m-1$. An inspection of \eqref{Kehsdbg-THM} reveals that $k_{\lambda\beta}$ belongs to 
$\mathscr{C}^{\infty}(\mathbb{R}^{n}\setminus\{0\})$, is even, and positive homogeneous of degree $-n$. 
Let us check that $z\,k_{\lambda\beta}(z)$ is a null-solution of $\Delta^{m}$ in $\mathbb{R}^{n}\setminus\{0\}$. 
Given \eqref{Kehsdbg-THM}, it suffices to show that $z$ times each summand in $k_{\lambda\beta}(z)$ has this property. 
That is, we need to prove that for each $\mu\in{\mathbb{N}}_0^n$ with $\mu\leq\lambda\leq\mu+\beta$ and each $j\in\{1,\dots,n\}$, we have 
\begin{equation}\label{Pihbn}
\Delta^m_z\Bigg(z^{\beta-\lambda+\mu+e_j}\partial^{\mu}_z\bigg[\frac{1}{|z|^{n}}\bigg]\Bigg)=0\,
\text{ in }\,\mathbb{R}^{n}\setminus\{0\}.
\end{equation}
The latter is a consequence of the following claim:
\begin{equation}\label{bGDSWw}
\begin{array}{c}
\text{if $\mu,\omega\in{\mathbb{N}}_0^n$ satisfy $|\mu|\leq m-1$ and $|\omega|=|\mu|+1$}
\\[4pt]
\displaystyle
\text{then 
$\Delta_z^m\Bigg(z^\omega\partial_z^\mu\bigg[\frac{1}{|z|^{n}}\bigg]\Bigg)=0$ in $\mathbb{R}^{n}\setminus\{0\}$.
}
\end{array}
\end{equation}
To see why \eqref{bGDSWw} is true, fix $\mu,\omega$ as in \eqref{bGDSWw}, pick $j\in{\rm supp}\,\omega$ and set 
$\widetilde{\omega}:=\omega-e_j$. Then Leibniz's formula gives
\begin{equation}\label{Ksjdhfgksjdhg.bb}
z^\omega\,\partial^{\mu}_z\bigg[\frac{1}{|z|^{n}}\bigg]
=z^{\widetilde{\omega}}\partial^{\mu}_z\bigg[\frac{z_{j}}{|z|^{n}}\bigg]
-\mu_{j}z^{\widetilde{\omega}}\,\partial^{\mu-e_{j}}_z\bigg[\frac{1}{|z|^{n}}\bigg]
\,\,\,\,\text{in $\mathbb{R}^{n}\setminus\{0\}$.}
\end{equation}
Now \eqref{bGDSWw} follows from \eqref{Ksjdhfgksjdhg.bb} and \eqref{HSjhagsdfjhgsgf} by induction on $|\mu|$ upon noting that 
$z^{\widetilde{\omega}}\in\mathcal{P}_{m-1}$ and $\frac{z_{j}}{|z|^{n}}$ is harmonic in $\mathbb{R}^{n}\setminus\{0\}$. 

To complete the proof of the theorem, it remains to verify that \eqref{eq:ugt-igFRD} also holds. Since this is a purely algebraic 
condition, independent of the underlying domain, we may strengthen the original assumptions by assuming that $\Omega$ is a {\rm UR} domain with compact boundary.
Denote $\Omega_{+}:=\Omega$ and $\Omega_{-}:=\mathbb{R}^{n}\setminus\overline{\Omega}$ and fix $\lambda\in\mathbb{N}_{0}^{n}$ 
with $|\lambda|=m-1$ along with some $P\in\mathcal{P}_{m-1}$. On the one hand, \eqref{HShagdfsdg-THM} and \eqref{Thm-MSafddfgdsa}
allow us to invoke \cite[(5.2.138), p.\,624]{GHA.IV} to obtain
\begin{equation}\label{Qwefhsdhfihgh}
(\partial^{\lambda}\dot{\mathcal{D}}_{\Delta^{m}}\dot{P})\big|_{\partial\Omega_{+}}^{{}^{\kappa\text{\rm -nt}}}
-(\partial^{\lambda}\dot{\mathcal{D}}_{\Delta^{m}}\dot{P})\big|_{\partial\Omega_{-}}^{{}^{\kappa\text{\rm -nt}}}
=\sum_{|\beta|=m-1}\big(\partial^{\beta}P\big)\int_{S^{n-1}}k_{\lambda\beta}\,d\mathcal{H}^{n-1}
\end{equation}
at $\sigma$-a.e. point on $\partial\Omega$. On the other hand, Proposition~\ref{MAahsifuhfiouh} ensures that
\begin{equation}\label{QWyshdfghjg}
(\partial^{\lambda}\dot{\mathcal{D}}_{\Delta^{m}}\dot{P})\big|_{\partial\Omega_{+}}^{{}^{\kappa\text{\rm -nt}}}
-(\partial^{\lambda}\dot{\mathcal{D}}_{\Delta^{m}}\dot{P})\big|_{\partial\Omega_{-}}^{{}^{\kappa\text{\rm -nt}}}
=\partial^{\lambda}P
\quad\text{at $\sigma$-a.e. point on $\partial\Omega$.}
\end{equation}
Comparing \eqref{Qwefhsdhfihgh} with \eqref{QWyshdfghjg} yields
\begin{equation}\label{Uyugsgshdgf}
\partial^{\lambda}P=\sum_{|\beta|=m-1}\big(\partial^{\beta}P\big)\int_{S^{n-1}}k_{\lambda\beta}\,d\mathcal{H}^{n-1}
\quad\text{at $\sigma$-a.e. point on $\partial\Omega$.}
\end{equation}
Specializing \eqref{Uyugsgshdgf} to the case when $P(x):=x^{\lambda}$ forces 
$\int_{S^{n-1}}k_{\lambda\beta}\,d\mathcal{H}^{n-1}=\delta_{\lambda\beta}$, as wanted.
\end{proof}

\begin{lemma}\label{NShhgsfhgiguh}
Fix $n\in\mathbb{N}$ with $n\geq 2$, and recall that $E_\Delta$ stands for the canonical fundamental solution for the Laplacian in $\mathbb{R}^{n}$. 
For each $\gamma\in\mathbb{N}_{0}^{n}$, define the vector-valued function
\begin{equation}\label{JSKLasdglskgkg}
\vec{F}_{\gamma}(x):=\frac{1}{\omega_{n-1}}\frac{x}{|x|^{n}}\frac{x^{\gamma}}{\gamma!}
=(\nabla E_\Delta)(x)\,\frac{x^{\gamma}}{\gamma!}\,\,\text{ for every }\,\,x\in\mathbb{R}^{n}\setminus\{0\}.
\end{equation}

Then the following are true.
\begin{enumerate}
\item[(i)] Given any $\gamma\in\mathbb{N}_{0}^{n}$, one has 
$\vec{F}_{\gamma}\in[\mathscr{C}^{\infty}(\mathbb{R}^{n}\setminus\{0\})]^{n}$, the components of $\vec{F}_{\gamma}$ are positive homogeneous 
of degree $1-n+|\gamma|$, and $\Delta^{m}\vec{F}_{\gamma}=0$ in $\mathbb{R}^{n}\setminus\{0\}$ whenever $|\gamma|\leq m-1$.

\vskip 0.08in
\item[(ii)] For each $\gamma\in\mathbb{N}_{0}^{n}$ and each multi-index $\alpha\in\mathbb{N}_{0}^{n}$ satisfying 
$|\alpha|\geq|\gamma|+1-n$ one has
\begin{equation}\label{SAakdfgsdgh}
\big|\partial^{\alpha}\vec{F}_{\gamma}(x)\big|
\lesssim\frac{1}{1+|x|^{n-1-|\gamma|+|\alpha|}}
\,\,\,\,\text{uniformly for $x\in\mathbb{R}^{n}$ away from the origin.}
\end{equation}
\end{enumerate}
\end{lemma}

\begin{proof}
That $\Delta^{m}\vec{F}_{\gamma}=0$ in $\mathbb{R}^{n}\setminus\{0\}$ whenever $|\gamma|\leq m-1$ is seen from \eqref{JSKLasdglskgkg}
and \eqref{HSjhagsdfjhgsgf}. The justification of all other claims is standard.
\end{proof}

\begin{lemma}\label{Propsdewhfgoejhg} 
Let $n\in\mathbb{N}$ with $n\geq 2$ and fix a reference point $x_{\ast}\in\mathbb{R}^{n}$.
For each $\gamma\in\mathbb{N}_{0}^{n}$ set 
\begin{equation}\label{Qwfghjfhjf}
\begin{array}{c}
\displaystyle\vec{H}_{\gamma}(x,y):=\vec{F}_{\gamma}(x-y)
-{\mathbb{P}}_{m-2,x_\ast}\big[\vec{F}_{\gamma}(\cdot-y)\big](x)
\\ \noalign{\vskip 4pt}
\text{for all points $x,y\in\mathbb{R}^{n}$ such that $y\notin\{x,x_\ast\}$}
\end{array}
\end{equation}
where $\vec{F}_{\gamma}$ is as in \eqref{JSKLasdglskgkg} and ${\mathbb{P}}_{m-2,x_\ast}$ is the 
functor defined in \eqref{mxcvbprm3fdh.PP-INTRO} with $\ell:=m-2$. 

Then for each $\gamma\in\mathbb{N}_{0}^{n}$ with $|\gamma|\leq m-1$ and each $y\in\mathbb{R}^{n}\setminus\{x_{\ast}\}$, 
\begin{equation}\label{Mjashfgshjjj}
\Delta^{m}\big[\vec{H}_{\gamma}(\cdot,y)\big]=0\,\text{ in }\,\mathbb{R}^{n}\setminus\{y\}.
\end{equation}

Also, if $\gamma\in\mathbb{N}_{0}^{n}$ has $|\gamma|\leq n+m-2$, then for each $\alpha\in\mathbb{N}_{0}^{n}$ and $x\in\mathbb{R}^{n}$ one has
\begin{equation}\label{Wfwgdh}
\begin{array}{c}
\displaystyle\Big|\partial_{x}^{\alpha}\big[\vec{H}_{\gamma}(x,y)\big]\Big|
\lesssim_{x,x_{\ast}}\frac{1}{\big(1+|y|\big)^{n-1-|\gamma|+\max\{m-1,\,|\alpha|\}}}
\\ \noalign{\vskip 6pt}
\text{uniformly for $y\in\mathbb{R}^{n}$ away from $x$ and $x_{\ast}$.}
\end{array}
\end{equation}
\end{lemma}

\begin{proof} 
The identity in \eqref{Mjashfgshjjj} is a consequence of item {\it (i)} in Lemma~\ref{NShhgsfhgiguh} and the 
fact that ${\mathbb{P}}_{m-2,x_\ast}\big[\vec{F}_{\gamma}(\cdot-y)\big]$ is a polynomial of degree $\leq m-2$.
To justify \eqref{Wfwgdh}, fix $\gamma\in\mathbb{N}_{0}^{n}$ with $|\gamma|\leq n+m-2$. Also, let $\alpha\in\mathbb{N}_{0}^{n}$ 
and $x\in\mathbb{R}^{n}$ be arbitrary, then consider $y\in\mathbb{R}^{n}$ away from $x$ and $x_{\ast}$. The definition in \eqref{Qwfghjfhjf} entails 
\begin{equation}\label{Qwfghjfhjf-nn}
\partial_{x}^{\alpha}\big[\vec{H}_{\gamma}(x,y)\big]=(\partial^{\alpha}\vec{F}_{\gamma})(x-y)
-{\mathbb{P}}_{m-2-|\alpha|,x_\ast}\big[(\partial^{\alpha}\vec{F}_{\gamma})(\cdot-y)\big](x).
\end{equation}
If $|\alpha|\geq m-1$ then the last term above drops out, and the desired estimate is provided by item {\it (ii)} in Lemma~\ref{NShhgsfhgiguh}.
Consider now $|\alpha|\leq m-2$. Adjusting constants, it suffices to treat the case when $y$ stays away from the line segment $[x,x_\ast]$. 
Then from \eqref{Qwfghjfhjf-nn} and Taylor's formula with integral remainder we see that 
\begin{align}\label{Qwfghjfhjf-nn.2}
\partial_{x}^{\alpha}\big[\vec{H}_{\gamma}(x,y)\big] &=\sum_{|\beta|=m-1-|\alpha|}\frac{m-1-|\alpha|}{\beta!}
\int_{0}^{1}(1-t)^{m-2-|\alpha|}(x-x_{\ast})^{\beta}\times 
\nonumber\\
&\qquad\qquad\times\big(\partial^{\alpha+\beta}\vec{F}_{\gamma}\big)\big(tx+(1-t)x_{\ast}-y\big)\,dt, 
\end{align}
and the desired estimate is once again provided by item {\it (ii)} in Lemma~\ref{NShhgsfhgiguh}.
\end{proof}

\begin{definition}\label{modified-mahdhasd} 
Fix $n,m\in\mathbb{N}$ with $n\geq 2$ and suppose $\Omega\subseteq\mathbb{R}^{n}$ is an Ahlfors regular domain. Abbreviate 
$\sigma:=\mathcal{H}^{n-1}\lfloor{\partial\Omega}$ and denote by $\nu$ the geometric measure theoretic outward unit normal to 
$\Omega$. Having fixed a point $x_\ast\in{\mathbb{R}}^n\setminus\partial\Omega$, 
recall the piece of notation introduced in \eqref{Qwfghjfhjf}. In this context, for each family 
$\dot{f}=\{f_{\gamma}\}_{|\gamma|\leq m-1}$ of complex-valued functions with the property that
\begin{equation}\label{QWdfwsfdfg}
f_\gamma\in L^1\bigg(\partial\Omega,\frac{\sigma(x)}{1+|x|^{n+m-2-|\gamma|}}\bigg)
\,\,\text{ for each $\gamma\in{\mathbb{N}}_0^n$ with $|\gamma|\leq m-1$}, 
\end{equation}
define 
\begin{align}\label{modfeiugfh}
(\dot{\mathfrak{D}}_{\Delta^{m}}\,\dot{f})(x):=-\sum_{|\gamma|\leq m-1}\,\,\,
\int\limits_{\partial\Omega}\nu(y)\cdot\vec{H}_{\gamma}(x,y)\, f_{\gamma}(y)d\sigma(y)\,\text{ for each }\,x\in\mathbb{R}^{n}\setminus\partial\Omega.
\end{align}
\end{definition}

Call $\dot{\mathfrak{D}}_{\Delta^{m}}$ the {\tt boundary}-{\tt to}-{\tt domain} {\tt modified} {\tt multi}-{\tt layer} 
{\tt potential} {\tt operator} associated with the polylaplacian $\Delta^m$ and the set $\Omega$. For each family $\dot{f}=\{f_\gamma\}_{|\gamma|\leq m-1}$ 
of functions as in \eqref{QWdfwsfdfg}, the estimate in \eqref{Wfwgdh} ensures that $\dot{\mathfrak{D}}_{\Delta^{m}}\dot{f}$ is well defined via 
absolutely convergent integrals. Moreover, one may differentiate $\dot{\mathfrak{D}}_{\Delta^{m}}\dot{f}$ arbitrarily many times under 
the integral sign, since this only improves the decay of the integral kernel. Together with \eqref{Mjashfgshjjj}, this shows that
\begin{equation}\label{modified-multiprojsdg}
\dot{\mathfrak{D}}_{\Delta^{m}}\dot{f}\in{\mathscr{C}}^\infty(\mathbb{R}^{n}\setminus\partial\Omega)\,\,\text{ and }\,\,
\Delta^{m}(\dot{\mathfrak{D}}_{\Delta^{m}}\dot{f})=0\,\,\text{ in }\,\,\mathbb{R}^{n}\setminus\partial\Omega.
\end{equation}

\begin{theorem}\label{Jksjadfhsgh} 
Fix $n,m\in\mathbb{N}$ with $n\geq 2$. Let $\Omega\subseteq\mathbb{R}^{n}$ be an Ahlfors regular domain, denote by $\nu$ its 
geometric measure theoretic outward unit normal, and abbreviate $\sigma:=\mathcal{H}^{n-1}\lfloor{\partial\Omega}$. In this context, 
let $\dot{\mathfrak{D}}_{\Delta^{m}}$ be the boundary-to-domain modified multi-layer potential operator associated with the set $\Omega$ 
and the polylaplacian $\Delta^m$ as in Definition~\ref{modified-mahdhasd}. Finally, bring in the family 
$\{k_{\lambda\gamma}\}_{|\lambda|=|\gamma|=m-1}$ of functions from \eqref{Kehsdbg-THM} which satisfies 
\eqref{HShagdfsdg-THM}-\eqref{eq:ugt-igFRD} {\rm(}cf. Theorem~\ref{Thjdsbgkjbsdgkbh}{\rm )}.

Then for each family $\dot{f}=\{f_{\gamma}\}_{|\gamma|\leq m-1}$ of complex-valued functions with the property that
\begin{equation}\label{Wdsdghhjjj}
\dot{f}\in({\rm CC})_{m-1}\,\,\text{ and }\,\,f_\gamma\in L^1\bigg(\partial\Omega,\frac{\sigma(x)}{1+|x|^{n+m-2-|\gamma|}}\bigg)
\,\,\text{ for each $\gamma\in{\mathbb{N}}_0^n$ with $|\gamma|\leq m-1$}, 
\end{equation}
each multi-index $\lambda\in\mathbb{N}_{0}^{n}$ with $|\lambda|=m-1$, and at each point 
$x\in\mathbb{R}^{n}\setminus\partial\Omega$, one has
\begin{equation}\label{Sadgdsfdfga}
\partial^{\lambda}\big(\dot{\mathfrak{D}}_{\Delta^{m}}\dot{f}\big)(x)
=\sum_{|\gamma|=m-1}\int_{\partial\Omega}\nu(y)\cdot(y-x)\, k_{\lambda\gamma}(x-y)\,f_{\gamma}(y)\,d\sigma(y).
\end{equation}
\end{theorem}

\begin{proof}
Fix a cut-off function $\theta\in\mathscr{C}_c^\infty(\mathbb{R}^n)$ with $0\leq\theta\leq1$, ${\rm supp}\,\theta\subseteq B(0,2)$,
and $\theta\equiv1$ on $B(0,1)$. For each $R>0$ define $\theta_R(x):=\theta(x/R)$ for every $x\in\mathbb{R}^{n}$.
Then $\theta_R\in\mathscr{C}_c^\infty(\mathbb{R}^n)$ with $0\leq\theta_{R}\leq 1$, $\operatorname{supp}\,\theta_R\subseteq B(0,2R)$, 
$\theta_R\equiv1$ on $B(0,R)$, and
\begin{equation}\label{Dfedfnjsdnf}
\big|\partial^\alpha\theta_R(x)\big|\lesssim R^{-|\alpha|}\,{\mathbf{1}}_{|x|\sim R},\,
\text{ for every }\,\alpha\in\mathbb{N}_0^n\,\text{ with }\,|\alpha|>0.
\end{equation}
Fix $\dot f=\{f_\gamma\}_{|\gamma|\leq m-1}$ as in \eqref{Wdsdghhjjj} and for each $R>0$ 
define $\dot{f}^{R}:=\big\{f^R_\gamma\big\}_{|\gamma|\leq m-1}$ by setting
\begin{equation}\label{Mshfggfgtdsdfgx}
f_{\gamma}^{R}:=\theta_{R}f_\gamma+\sum_{\substack{\eta+\delta=\gamma\\ |\eta|>0}}
\frac{\gamma!}{\eta!\,\delta!}(\partial^\eta\theta_R)f_{\delta}\,\text{ for each }\,|\gamma|\leq m-1.
\end{equation}
Observe that \eqref{Wdsdghhjjj} and the properties of $\theta_{R}$ guarantee 
$\dot{f}^{R}\in{\rm WA}_{m-1}[L_{\rm comp}^{1}(\partial\Omega,\sigma)]$ (recall \eqref{WAajasnfksjdnf}).
To proceed, fix $\lambda\in\mathbb{N}_0^n$ with $|\lambda|=m-1$, along with some point $x\in\mathbb{R}^{n}\setminus\partial\Omega$. 
Differentiating under the integral sign allows us to write 
\begin{align}\label{ASdwfwewq}
\partial^\lambda\big(\dot{\mathfrak{D}}_{\Delta^{m}}\dot{f}\big)(x)&=-\int_{\partial\Omega}\sum_{|\gamma|\leq m-1}\nu(y)\cdot
\partial_{x}^\lambda[\vec{H}_{\gamma}(x,y)]\,f_\gamma(y)\,d\sigma(y)\nonumber\\
&=-\int_{\partial\Omega}\sum_{|\gamma|\leq m-1}\nu(y)\cdot
(\partial^\lambda\vec{F}_{\gamma})(x-y)\,f_\gamma(y)\,d\sigma(y).
\end{align}
We make the claim that
\begin{equation}\label{Sccdswedc}
\partial^\lambda\big(\dot{\mathfrak{D}}_{\Delta^{m}}\dot{f}\big)(x)=-\lim_{R\to\infty}\int_{\partial\Omega}\sum_{|\gamma|\leq m-1}
\nu(y)\cdot(\partial^\lambda\vec{F}_{\gamma})(x-y)f_\gamma^R(y)\,d\sigma(y).
\end{equation}
Indeed, using the formula for $f_{\gamma}^{R}$ in \eqref{Mshfggfgtdsdfgx}, it is enough to show that 
\begin{align}\label{Msjhfdfsh}
&\lim_{R\to\infty}\int_{\partial\Omega}\sum_{|\gamma|\leq m-1}\nu(y)\cdot
(\partial^\lambda\vec{F}_{\gamma})(x-y)\,\theta_{R}(y)f_\gamma(y)\,d\sigma(y)
\nonumber \\[4pt]
&\qquad=\int_{\partial\Omega}\sum_{|\gamma|\leq m-1}\nu(y)\cdot
(\partial^\lambda\vec{F}_{\gamma})(x-y)\,f_\gamma(y)\,d\sigma(y),
\end{align}
and  
\begin{equation}\label{SAdfsfghbdfbh}
\lim_{R\to\infty}\int_{\partial\Omega}\sum_{|\gamma|\leq m-1}\,\sum_{\substack{\eta+\delta=\gamma\\|\eta|>0}}
\frac{\gamma!}{\eta!\,\delta!}\,\nu(y)\cdot(\partial^\lambda\vec F_\gamma)(x-y)\,(\partial^\eta\theta_R)(y)f_\delta(y)\,d\sigma(y)=0.
\end{equation}
To justify \eqref{Msjhfdfsh}, recall that $0\leq\theta_{R}\leq 1$ and bring in \eqref{SAakdfgsdgh} (with $\alpha:=\lambda$) 
to conclude that the integrand in the first line of \eqref{Msjhfdfsh} is bounded, uniformly with respect to $R$, by
\begin{equation}\label{Nsdfhsdhfgh}
\sum_{|\gamma|\leq m-1}\big|(\partial^\lambda\vec{F}_{\gamma})(x-\cdot)\big|\,|f_\gamma|
\lesssim\sum_{|\gamma|\leq m-1}\frac{|f_\gamma|}{1+|\cdot|^{n+m-2-|\gamma|}}\in L^1(\partial\Omega,\sigma),
\end{equation}
with the membership in \eqref{Nsdfhsdhfgh} guaranteed by \eqref{Wdsdghhjjj}. In view of the fact that $\lim\limits_{R\to\infty}\theta_R=1$ 
pointwise in $\mathbb{R}^{n}$, Lebesgue Dominated Convergence Theorem applies and gives \eqref{Msjhfdfsh}.

Next, \eqref{SAdfsfghbdfbh} also follows by invoking the Lebesgue Dominated Convergence Theorem. Indeed, the pointwise 
convergence to zero of the integrand in \eqref{SAdfsfghbdfbh} is a consequence of the support condition on $\partial^\eta\theta_R$. 
As for the uniform bound, for each $\gamma\in{\mathbb{N}}_0^n$ with $|\gamma|=m-1$, once again bring in \eqref{SAakdfgsdgh} 
(with $\alpha:=\lambda$) and \eqref{Dfedfnjsdnf} to estimate
\begin{align}\label{Sfwegerf}
\sum_{\substack{\eta+\delta=\gamma\\|\eta|>0}}\left|
(\partial^\lambda\vec{F}_{\gamma})(x-y)
(\partial^\eta\theta_R)(y)f_\delta(y)\right|
&\lesssim\sum_{\substack{\eta+\delta=\gamma\\|\eta|>0}}\frac{1}{1+|y|^{n+m-2-|\gamma|}}R^{-|\eta|}
{\mathbf{1}}_{|y|\sim R}|f_\delta(y)|
\nonumber\\[4pt]
&\lesssim{\mathbf{1}}_{|y|\sim R}\sum_{|\delta|\leq m-1}\frac{|f_\delta(y)|}{1+|y|^{n+m-2-|\delta|}}
\end{align}
for $\sigma$-a.e. $y\in\partial\Omega$. Recalling that $\dot{f}$ satisfies \eqref{Wdsdghhjjj}, we obtain that the sum on 
the last line of \eqref{Sfwegerf} belongs to $L^1(\partial\Omega,\sigma)$. This finishes the justification of \eqref{Sccdswedc}.

Moving on, since each $f_\gamma^R$ has compact support, we may write \eqref{Sccdswedc} as
\begin{equation}\label{Sfsghdfg}
\partial^\lambda\big(\dot{\mathfrak{D}}_{\Delta^{m}}\dot f\big)(x)
=-\lim_{R\to\infty}\bigg(\partial_x^\lambda\bigg[\int_{\partial\Omega}
\sum_{|\gamma|\leq m-1}\nu(y)\cdot\vec{F}_{\gamma}(x-y)\,f_\gamma^R(y)\,d\sigma(y)\bigg]\bigg).
\end{equation}
Bearing in mind \eqref{Qlsdfpgjsdoigj} and \eqref{JSKLasdglskgkg}, as well as the fact that 
$\theta_{R}\dot{f}\in{\rm WA}_{m-1}[L_{\rm comp}^{1}(\partial\Omega,\sigma)]$, we conclude from \eqref{Sfsghdfg} that
\begin{equation}\label{Edfswfewdf}
\partial^\lambda\big(\dot{\mathfrak{D}}_{\Delta^{m}}\dot f\big)(x)
=\lim_{R\to\infty}\partial^\lambda\big({\dot{\mathcal{D}}}_{\Delta^m}\dot{f}^{R}\big)(x).
\end{equation}
Therefore, it follows from Theorem~\ref{Thjdsbgkjbsdgkbh} that
\begin{equation}\label{Dwdsffdgg}
\partial^\lambda\big(\dot{\mathfrak{D}}_{\Delta^{m}}\dot f\big)(x)=\lim_{R\to\infty}\sum_{|\gamma|=m-1}\int_{\partial\Omega}
\nu(y)\cdot (y-x)\, k_{\lambda\gamma}(x-y)f_{\gamma}^{R}(y)\,d\sigma(y).
\end{equation}
Finally, we claim that
\begin{align}\label{Nshsdwsd}
&\lim_{R\to\infty}\sum_{|\gamma|=m-1}\int_{\partial\Omega}\nu(y)\cdot (y-x)\,k_{\lambda\gamma}(x-y)f_{\gamma}^{R}(y)\,d\sigma(y)
\nonumber\\
&\qquad=\sum_{|\gamma|=m-1}\int_{\partial\Omega}\nu(y)\cdot (y-x)\, k_{\lambda\gamma}(x-y)f_{\gamma}(y)\,d\sigma(y).
\end{align}
To verify \eqref{Nshsdwsd}, bring in \eqref{Mshfggfgtdsdfgx} and note that it suffices to show that
\begin{align}\label{sdffwedsf}
&\lim_{R\to\infty}\int_{\partial\Omega}\sum_{|\gamma|=m-1}\nu(y)\cdot
(y-x)\,k_{\lambda\gamma}(x-y)\,\theta_R(y)f_\gamma(y)\,d\sigma(y)
\nonumber\\[6pt]
&\qquad =\int_{\partial\Omega}\sum_{|\gamma|=m-1}\nu(y)\cdot(y-x)\, k_{\lambda\gamma}(x-y)f_{\gamma}(y)\,d\sigma(y)
\end{align}
as well as
\begin{align}\label{Cewcdfdsdf}
&\lim_{R\to\infty}\int_{\partial\Omega}\sum_{|\gamma|=m-1}\,\sum_{\substack{\eta+\delta=\gamma\\|\eta|>0}}
\frac{\gamma!}{\eta!\,\delta!}\,\nu(y)\cdot(y-x)\,k_{\lambda\gamma}(x-y)\,(\partial^\eta\theta_R)(y)\,f_\delta(y)\,d\sigma(y)=0.
\end{align}
Both are consequences of the Lebesgue Dominated Convergence Theorem. Indeed, as before the pointwise convergence follows from 
the properties of $\theta_R$. As for the domination uniform with respect to $R$, for the integrand in the first line of 
\eqref{sdffwedsf} we have
\begin{equation}\label{Sfsdfgsdfg}
\sum_{|\gamma|=m-1}\big|x-\cdot\big|\big|k_{\lambda\gamma}(x-\cdot)\big|\big|\,|\theta_R||f_\gamma|
\lesssim_{x}\sum_{|\gamma|=m-1}\frac{|f_\gamma|}{1+|\cdot|^{n-1}}\in L^1(\partial\Omega,\sigma),
\end{equation}
with the membership ensured by having $f_{\gamma}\in L^{1}\Big(\partial\Omega,\frac{\sigma(y)}{1+|y|^{n-1}}\Big)$ 
for every multi-index $\gamma\in\mathbb{N}_{0}^{n}$ satisfying $|\gamma|=m-1$. This proves \eqref{sdffwedsf}.

Regarding \eqref{Cewcdfdsdf}, note that for each $\gamma\in{\mathbb{N}}_0^n$ with $|\gamma|=m-1$ we have
(relying also on \eqref{Dfedfnjsdnf})
\begin{align}\label{Ddffhdhgbdwsgdf}
\sum_{\substack{\eta+\delta=\gamma\\|\eta|>0}}
\big|y-x\big| \big|k_{\lambda\gamma}(x-y)\big|\big|(\partial^\eta\theta_R)(y)\big| \big|f_\delta(y)\big|
&\lesssim_{x}\sum_{\substack{\eta+\delta=\gamma\\|\eta|>0}}\frac{1}{1+|y|^{n-1}}R^{-|\eta|}
{\mathbf{1}}_{|y|\sim R}|f_\delta(y)|
\nonumber\\[4pt]
&\lesssim{\mathbf{1}}_{|y|\sim R}\sum_{|\delta|\leq m-1}\frac{|f_\delta(y)|}{1+|y|^{n+m-2-|\delta|}}
\end{align}
for $\sigma$-a.e. $y\in\partial\Omega$. Once again, having $\dot{f}$ as in \eqref{Wdsdghhjjj} guarantees that the last 
line of \eqref{Ddffhdhgbdwsgdf} belongs to $L^1(\partial\Omega,\sigma)$. Therefore, the Lebesgue Dominated Convergence Theorem applies 
and gives \eqref{Cewcdfdsdf}. This finishes the justification of \eqref{Nshsdwsd}. In turn, \eqref{Nshsdwsd} and 
\eqref{Dwdsffdgg} yield \eqref{Sadgdsfdfga}.
\end{proof}

Recall that given $\Omega\subseteq{\mathbb{R}}^n$ we denote $\Omega_{+}=\Omega$ and 
$\Omega_{-}={\mathbb{R}}^n\setminus\overline{\Omega}$.

\begin{proposition}\label{prop:claim}
Let $n,m\in\mathbb{N}$ with $n\geq 2$ and suppose $\Omega\subseteq\mathbb{R}^{n}$ is an Ahlfors regular domain with unbounded boundary. 
Abbreviate $\sigma:=\mathcal{H}^{n-1}\lfloor{\partial\Omega}$, and denote by $\nu$ its geometric measure theoretic outward unit normal. 
Also, fix a reference point $x_{\ast}\in\mathbb{R}^{n}\setminus\partial\Omega$. Let $\dot{\mathfrak{D}}_{\Delta^{m}}$ be the boundary-to-domain 
modified multi-layer potential operator associated with the set $\Omega$ and the point $x_\ast$ as in Definition~\ref{modified-mahdhasd}.
Then for every $P\in{\mathcal{P}}_{m-2}$ and every $x\in{\mathbb{R}}^n\setminus\partial\Omega$ one has
\begin{equation}\label{eq:claim}
(\dot{\mathfrak{D}}_{\Delta^{m}}\dot{P})(x)
=\begin{cases}
\pm P(x), & \text{if }\,x\in\Omega_\pm\,\text{ and }\,x_*\in\Omega_\mp,
\\[4pt]
0, & \text{if }\,x,x_*\in\Omega_{+},\,\text{ or }\,x,x_*\in\Omega_{-}.
\end{cases}
\end{equation}
\end{proposition}

\begin{proof}
Fix $x\in{\mathbb{R}}^n\setminus\partial\Omega$ and let $P\in{\mathcal{P}}_{m-2}$. Recall 
$\vec{F}_\gamma$ from \eqref{JSKLasdglskgkg} and the functor from \eqref{mxcvbprm3fdh.PP-INTRO}. Since $\deg P\leq m-2$,
formula \eqref{eq:P-reproduces} gives
\begin{equation}\label{eqBNHJg}
\sum_{|\gamma|\leq m-1}\vec{F}_\gamma(x-y)\,(\partial^\gamma P)(y)
=(\nabla E_\Delta)(x-y)\,{\mathbb{P}}_{m-1,y}[P](x)=(\nabla E_\Delta)(x-y)\,P(x)
\end{equation}
for every $y\in{\mathbb{R}}^n\setminus\{x\}$. This and the linearity of ${\mathbb{P}}_{m-2,x_*}[\,\cdot\,]$ then yield
\begin{align}\label{eq:aux-lin}
\sum_{|\gamma|\leq m-1}{\mathbb{P}}_{m-2,x_*}\big[\vec{F}_\gamma(\cdot-y)\big](x)\,(\partial^\gamma P)(y) 
&={\mathbb{P}}_{m-2,x_*}\Big[\sum_{|\gamma|\leq m-1}
\vec{F}_\gamma(\cdot-y)(\partial^\gamma P)(y)\Big](x) 
\nonumber\\[4pt]
&={\mathbb{P}}_{m-2,x_*}\big[(\nabla E_\Delta)(\cdot-y)\,P(\cdot)\big](x)
\end{align}
for every $y\in{\mathbb{R}}^n\setminus\{x\}$. Invoking \eqref{eq:2bis} with $w(\cdot):=(\nabla E_\Delta)(\cdot-y)$ 
(with $y$ fixed) for the right-hand side of \eqref{eq:aux-lin}, we obtain
\begin{align}\label{eq:aux-2bis}
\sum_{|\gamma|\leq m-1}{\mathbb{P}}_{m-2,x_*}&\big[\vec{F}_\gamma(\cdot-y)\big](x)\,(\partial^\gamma P)(y)
=(\nabla E_\Delta)(x_*-y)\,P(x) 
\nonumber\\[4pt]
&+\sum_{\substack{|a|+|b|\leq m-2\\|a|>0}}\frac{(x-x_*)^a}{a!}(\nabla\partial^a E_\Delta)(x_*-y)
\frac{(x-x_*)^b}{b!}(\partial^b P)(x_*).
\end{align}
By combining \eqref{Qwfghjfhjf} and \eqref{modfeiugfh} (with $\dot f=\dot{P}$), \eqref{eqBNHJg}, and \eqref{eq:aux-2bis}, we obtain
\begin{align}\label{eq:before-split}
(\dot{\mathfrak{D}}_{\Delta^{m}}\dot{P})(x)=-\int_{\partial\Omega}\nu(y)\cdot\Big\{
&(\nabla E_\Delta)(x-y)P(x)-(\nabla E_\Delta)(x_*-y)P(x) \nonumber\\[4pt]
&-\sum_{\substack{|a|+|b|\leq m-2\\ |a|>0}}\frac{(x-x_*)^a}{a!}(\nabla\partial^a E_\Delta)(x_*-y)\times
\nonumber\\[0pt]
&\hspace{2.5cm}\times\frac{(x-x_*)^b}{b!}(\partial^b P)(x_*)\Big\}\,d\sigma(y).
\end{align}
Write the right-hand side of \eqref{eq:before-split} as ${\rm I}(x)+{\rm II}(x)$, 
where ${\rm I}(x)$, ${\rm II}(x)$ are given, respectively, by
\begin{align}\label{eq:I-def}
{\rm I}(x)&:=\bigg(\int_{\partial\Omega}\nu(y)\cdot\vec{G}_{x,x_*}(y)\,d\sigma(y)\bigg)P(x), 
\\[4pt]
&\hspace{1cm}\text{with }\,\vec{G}_{x,x_*}:=\nabla\big[E_\Delta(\cdot-x)- E_\Delta(\cdot-x_*)\big],
\label{Gxxstar}
\end{align}
and
\begin{align}\label{eq:II-def}
{\rm II}(x)&:=\sum_{\substack{|a|+|b|\leq m-2\\ |a|>0}}\bigg(\int_{\partial\Omega}\nu(y)\cdot
\vec{J}_{x_*,a}(y)\,d\sigma(y)\bigg)\frac{(x-x_*)^a}{a!}\,\frac{(x-x_*)^b}{b!}(\partial^b P)(x_*),
\\[4pt]
&\hspace{3cm}\text{with }\,\vec{J}_{x_*,a}:=(-1)^{|a|+1}\nabla\partial^a\big[E_\Delta(\cdot-x_*)\big].
\label{Hxxstar}
\end{align}
Moreover, we note that the integrals in \eqref{eq:I-def} and \eqref{eq:II-def} are absolutely convergent, since
\begin{equation}\label{eq:estG}
|\vec{G}_{x,x_*}(y)|\lesssim_{x,x_*}\frac{1}{1+|y|^n}\,\text{ uniformly for $y\in{\mathbb{R}}^n$ away from $x,x_*$},
\end{equation}
\begin{equation}\label{eq:estH}
|\vec{J}_{x_*,a}(y)|\lesssim_{x_*}\frac{1}{1+|y|^{n-1+|a|}}\,\text{ uniformly for $y\in{\mathbb{R}}^n$ away from $x_*$}.
\end{equation}
The strategy now is to invoke the Divergence Theorem from  \cite[Theorem~1.4.1, pp.\,38--39]{GHA.I}, applied in $\Omega$. 
To this end, fix an arbitrary aperture parameter $\kappa\in(0,\infty)$.

We first focus on ${\rm I}(x)$. An inspection of \eqref{Gxxstar} combined with the properties of $E_\Delta$ reveals that 
$\vec{G}_{x,x_*}\in\big[\mathscr{C}^\infty(\overline\Omega\setminus\{x,x_*\})\cap 
L^1_{{\rm loc}}(\Omega,{\mathcal{L}}^n)\big]^n$ and, in the sense of distributions in $\Omega$,
\begin{align}\label{eq:divG}
{\rm div}\,\vec{G}_{x,x_*} &=\Delta\big[E_\Delta(\cdot-x)-E_\Delta(\cdot-x_*)\big]
=(\delta_x-\delta_{x_*})\big|_\Omega\in{\mathcal{E}}'(\Omega).
\end{align}
In particular, ${\rm div}\,\vec{G}_{x,x_*}\in{\mathcal{E}}'(\Omega)$ and
\begin{equation}\label{eq:divG-pairing}
{}_{{\mathcal{E}}'(\Omega)}\big\langle{\rm div}\,\vec{G}_{x,x_*},1\big\rangle_{{\mathcal{E}}(\Omega)}
=\begin{cases} 
0, & \text{ if }\,x,x_*\in\Omega_\pm, 
\\ 
\pm 1, & \text{ if }\,x\in\Omega_\pm \text{ and } x_*\in\Omega_\mp.  
\end{cases}
\end{equation}
By \eqref{eq:divG} and \cite[Lemma~8.3.7, (8.3.47), pp.\,693--694]{GHA.I}, for every compact set $K\subseteq\Omega$ such that 
$\Omega\setminus\mathring{K}$ contains neither $x$ nor $x_*$,
\begin{equation}\label{eq:G-nontang-est}
0\leq \big({\mathcal{N}}_\kappa^{\,\Omega\setminus K}\vec{G}_{x,x_*}\big)(y) \lesssim_{x,x_*,\kappa,K}
\frac{1}{1+|y|^n}\,\,\text{ for all }\,y\in\partial\Omega.
\end{equation}
Since ${\mathcal{N}}_\kappa^{\,\Omega\setminus K}\vec{G}_{x,x_*}$ is also $\sigma$-measurable (cf. \cite[(8.2.28), p.\,685]{GHA.I}), 
this places ${\mathcal{N}}_\kappa^{\,\Omega\setminus K}\vec{G}_{x,x_*}$ in $L^1(\partial\Omega,\sigma)$
(cf. \cite[(7.2.5), p.\,574]{GHA.I}). Granted this, the Divergence Theorem stated in \cite[Theorem~1.4.1, pp.\,38--39]{GHA.I}
together with \cite[(4.6.21), p.\,332]{GHA.I} apply and, in view of \eqref{eq:divG-pairing},
give
\begin{equation}\label{eq:I-jump}
{\rm I}(x)=\begin{cases} 
0, & \text{ if }\,x,x_*\in\Omega_{+},\,\text{ or }\,x,x_*\in\Omega_{-}
\\ 
\pm P(x), & \text{ if }\,x\in\Omega_\pm\,\text{ and }\,x_*\in\Omega_\mp. 
\end{cases}
\end{equation}

Moving on to ${\rm II}(x)$, since $E_\Delta$ is locally integrable and smooth outside the origin we have
$\vec{J}_{x_*,a}\in\big[{\mathcal{D}}'(\Omega)\big]^n$ and 
$\vec{J}_{x_*,a}\big|_{\Omega\setminus\{x_*\}}\in\big[\mathscr{C}^\infty(\overline\Omega\setminus\{x_*\})\big]^n$.
Also, for every compact set $K\subseteq\Omega$ such that $\Omega\setminus\mathring{K}$ does not contain $x_*$, 
it follows that ${\mathcal{N}}_\kappa^{\,\Omega\setminus K}\vec{J}_{x_*,a}$ is 
$\sigma$-measurable (cf. \cite[(8.2.28), p.\,685]{GHA.I}) and satisfies
\begin{equation}\label{eq:H-nontang-est}
0\leq\big({\mathcal{N}}_\kappa^{\,\Omega\setminus K}\vec{J}_{x_*,a}\big)(y)\lesssim_{x_*,\kappa,K}
\frac{1}{1+|y|^{n-1+|a|}}\,\text{ for all }\,y\in\partial\Omega.
\end{equation}
Hence, since $|a|>0$, we have ${\mathcal{N}}_\kappa^{\,\Omega\setminus K}\vec{J}_{x_*,a}\in L^1(\partial\Omega,\sigma)$. 
Moreover, from the definition of $\vec{J}_{x_*,a}$ in \eqref{Hxxstar}, it follows that, in the sense of distributions in $\Omega$,
\begin{align}\label{eq:div-H}
{\rm div}\,\vec{J}_{x_*,a} &=(-1)^{|a|+1}\Delta\partial^a\big[E_\Delta(\cdot-x_*)\big]
=(-1)^{|a|+1}\big(\partial^a\delta_{x_*}\big)\big|_\Omega \in {\mathcal{E}}'(\Omega).
\end{align}
Invoking again \cite[Theorem~1.4.1, pp.\,38--39]{GHA.I} and \cite[(4.6.21), p.\,332]{GHA.I}
(recalling that $\partial\Omega$ is unbounded), we conclude
\begin{align}\label{eq:II-zero}
{\rm II}(x) &=\sum_{\substack{|a|+|b|\leq m-2\\ |a|>0}}{}_{{\mathcal{E}}'(\Omega)} \big\langle{\rm div}\,\vec{J}_{x_*,a}\,,\,1
\big\rangle_{{\mathcal{E}}(\Omega)}\,\frac{(x-x_*)^{a+b}}{a!\,b!}(\partial^b P)(x_*)
\nonumber\\[4pt]
&=\sum_{\substack{|a|+|b|\leq m-2\\ |a|>0}}(-1)^{|a|+1}{}_{{\mathcal{E}}'(\Omega)}\big\langle\partial^a\delta_{x_*},1
\big\rangle_{{\mathcal{E}}(\Omega)}\,\frac{(x-x_*)^{a+b}}{a!\,b!}(\partial^b P)(x_*)=0,
\end{align}
since ${}_{{\mathcal{E}}'(\Omega)}\langle\partial^a\delta_{x_*},1\rangle_{{\mathcal{E}}(\Omega)}
=(-1)^{|a|}{}_{{\mathcal{E}}'(\Omega)}\langle\delta_{x_*},\partial^a1\rangle_{{\mathcal{E}}(\Omega)}=0$ whenever $|a|>0$.

Now \eqref{eq:claim} follows by combining \eqref{eq:before-split}, \eqref{eq:I-jump}, and \eqref{eq:II-zero}.
\end{proof}


\section{Proof of the Main result}

\begin{proof}[Proof of Theorem~\ref{DInRBVP-tk-INTRO}] 
Let $\delta_{\ast}\in(0,1)$ be the threshold given by 
\cite[Theorem~3.1.5, pp.\,118--119]{GHA.V}, depending only on $n$ and on the Ahlfors regularity character of $\partial\Omega$.
If $\|\nu\|_{{\rm BMO}(\partial\Omega,\sigma)}<\delta$ for some $\delta\in(0,\delta_{\ast}]$, the aforementioned result ensures 
that $\Omega$ is also a two-sided {\rm NTA} domain with an unbounded connected boundary. Henceforth, assume this to be the case.

Next, let ${\mathbb{X}}$ be a normed space belonging to the hierarchy of good spaces on $\partial\Omega$. As proved in
\cite[Proposition~2.48]{MMM.HOPT} the following continuous embedding holds
\begin{equation}\label{MKJydxd}
{\mathbb{X}}\hookrightarrow L^1\bigg(\partial\Omega\,,\,\frac{\sigma(x)}{1+|x|^{n-1}}\bigg).
\end{equation}
Fix an arbitrary $\dot{h}\in{\rm HWA}_{m-1}[\mathbb{X}]$ and set $w:=\dot{\mathfrak{D}}_{\Delta^{m}}\dot{h}$ in $\Omega$. 
Thanks to \eqref{HWA.X}, \eqref{MKJydxd}, \eqref{modified-multiprojsdg}, and Theorem~\ref{Jksjadfhsgh}, it follows that  
$w\in{\mathscr{C}}^\infty(\Omega)$ and at each $x\in\Omega$ we have
\begin{equation}\label{ddgdsdsdg}
(\nabla^{m-1}w)(x)=\Bigg(\sum_{|\beta|=m-1}
\int_{\partial\Omega}\nu(y)\cdot (y-x)\, k_{\lambda\beta}(x-y)h_{\beta}(y)\,d\sigma(y)\Bigg)_{|\lambda|=m-1}
\end{equation}
where the functions $\{k_{\lambda\beta}\}_{|\lambda|=|\beta|=m-1}$ are as in \eqref{Kehsdbg-THM}. As a consequence of Theorem~\ref{Thjdsbgkjbsdgkbh},
the family $\big\{k_{\lambda\beta}\big\}_{|\beta|=|\lambda|=m-1}$ has the properties listed in \eqref{HShagdfsdg-THM}.
As such, the components of the right-hand side of \eqref{ddgdsdsdg} are sums of boundary-to-domain integral operators to 
which \cite[Theorem~2.75]{MMM.HOPT} applies and gives that
\begin{equation}\label{Sfdfe}
(\partial^{\lambda}w)\big|_{\partial\Omega}^{{}^{\kappa\text{\rm -nt}}}\,
\text{ exists $\sigma$-a.e. on $\partial\Omega$ for each $|\lambda|=m-1$,}
\end{equation}
and
\begin{equation}\label{fwergerg}
\big\|\mathcal{N}_{\kappa}(\nabla^{m-1}w)\big\|_{\mathbb{X}}
\leq C\sum_{|\beta|=m-1}\|h_\beta\|_{{\mathbb{X}}}
= C\,\big\|\dot{h}\big\|_{{\rm HWA}_{m-1}[\mathbb{X}]},
\end{equation}
for some constant $C\in (0,\infty)$ independent of $\dot{h}\in{\rm HWA}_{m-1}[\mathbb{X}]$. Granted these properties, 
Theorem~\ref{thm:HWA-X.aaa} ensures that ${\rm Tr}^{{}^{\kappa\text{\rm-nt}}}_{m-1}w$ exists $\sigma$-a.e. on $\partial\Omega$, 
belongs to ${\rm HWA}_{m-1}[\mathbb{X}]$, and satisfies
\begin{equation}\label{BGrfdf}
\Big\|{\rm Tr}^{{}^{\kappa\text{\rm -nt}}}_{m-1}w\Big\|_{{\rm HWA}_{m-1}[{\mathbb{X}}]}
\leq C\,\big\|\dot{h}\big\|_{{\rm HWA}_{m-1}[\mathbb{X}]}.
\end{equation}

The above considerations permit us to define the linear and bounded operator
\begin{equation}\label{Sfeggfdgerg}
\begin{array}{c}
T:{\rm HWA}_{m-1}[\mathbb{X}]\longrightarrow{\rm HWA}_{m-1}[\mathbb{X}]
\\[4pt]
T\dot{h}:={\rm Tr}^{{}^{\kappa\text{\rm-nt}}}_{m-1}\big(\dot{\mathfrak{D}}_{\Delta^{m}}\dot{h}\big)-\tfrac{1}{2}\dot{h}
\,\text{ for each }\,\dot{h}\in{\rm HWA}_{m-1}[\mathbb{X}].
\end{array}
\end{equation}
Corresponding to $m=1$, the operator $T$ becomes the classical boundary-to-boundary harmonic double layer potential 
operator $K_{\Delta}$ (cf. \cite[(1.1.32), p.\,9]{GHA.I}). Also, observe that by design
\begin{equation}\label{TTTbchbcbgg} 
\big(\tfrac{1}{2}I+T\big)\dot{h}={\rm Tr}^{{}^{\kappa\text{\rm-nt}}}_{m-1} \big(\dot{\mathfrak{D}}_{\Delta^{m}}\dot{h}\big)
\,\text{ for every }\,\dot{h}\in{\rm HWA}_{m-1}[\mathbb{X}]. 
\end{equation} 

The next goal is to prove that, by eventually decreasing the threshold $\delta\in(0,\delta_{\ast}]$, the operator
\begin{equation}\label{Ttt-ISOMORPH}
\tfrac{1}{2}I+T:{\rm HWA}_{m-1}[\mathbb{X}]\longrightarrow 
{\rm HWA}_{m-1}[\mathbb{X}]\,\,\text{ is an isomorphism}.
\end{equation}
To show \eqref{Ttt-ISOMORPH}, first observe that by Proposition~\ref{prop:claim} (recall that 
$x_{\ast}\in\mathbb{R}^{n}\setminus\overline{\Omega}$) and \eqref{TTTbchbcbgg} one has
\begin{equation}\label{polynfujdfh}
\big(\tfrac{1}{2}I+T\big)\dot{P}=\dot{P}\,\,\,
\text{ for every $\dot{P}\in\dot{\mathcal{P}}_{m-2}$.}
\end{equation}
The latter and Theorem~\ref{thm:HWA-X} further imply that
$[T]:{\rm HWA}_{m-1}[\mathbb{X}]/\!\sim\,\to\,{\rm HWA}_{m-1}[\mathbb{X}]/\!\sim$ given
by $[T][\dot{h}]:=[T\dot{h}]$ for each $\dot{h}\in{\rm HWA}_{m-1}[\mathbb{X}]$ is a well defined, linear and bounded operator. 
Second, we claim that if $\|\nu\|_{{\rm BMO}(\partial\Omega,\sigma)}<\delta$ with $\delta\in(0,\delta_\ast]$ sufficiently small then the operator
\begin{equation}\label{Adwsfgg}
\tfrac{1}{2}I+[T]:{\rm HWA}_{m-1}\big[{\mathbb{X}}\big]\big/\sim\,\longrightarrow
{\rm HWA}_{m-1}\big[{\mathbb{X}}\big]\big/\sim\,\,\,\text{ is a Banach-space isomorphism}.
\end{equation}
To justify this claim, for each $\dot{h}\in{\rm HWA}_{m-1}[\mathbb{X}]$ use \eqref{Sfeggfdgerg} and \eqref{eq:EQUIV-CLASS.NNN} to write
\begin{equation}\label{mainquotnorm}
\Big\|[T][\dot{h}]\Big\|_{{\rm HWA}_{m-1}[\mathbb{X}]/{\sim}}
=\sum_{|\lambda|=m-1}\Big\|\big[\partial^{\lambda}\big(\dot{\mathfrak{D}}_{\Delta^{m}}\dot{h}\big)\big]
\big|_{\partial\Omega}^{{}^{\kappa\text{\rm -nt}}}-\tfrac{1}{2}h_{\lambda}\Big\|_{\mathbb{X}}.
\end{equation}
In concert, Theorem~\ref{Jksjadfhsgh} and the jump-formulas for the boundary traces of each
$\partial^{\lambda}(\dot{\mathfrak{D}}_{\Delta^{m}}\dot{h})$ with $|\lambda|=m-1$ 
(cf. \cite[Theorem~5.2.2, pp.\,623--641]{GHA.IV}, keeping in mind \eqref{MKJydxd} and the properties in 
\eqref{HShagdfsdg-THM}-\eqref{eq:ugt-igFRD}), imply that the right-hand side of \eqref{mainquotnorm} is bounded above by 
\begin{equation}\label{mainPV}
\sum_{|\beta|=|\lambda|=m-1}\bigg\|\lim_{\varepsilon\to 0^{+}}\int_{\substack{y\in\partial\Omega\\|\cdot-y|>\varepsilon}}
\nu(y)\cdot(y-\cdot)\,k_{\lambda\beta}(\cdot-y)h_{\beta}(y)\,d\sigma(y)\bigg\|_{\mathbb{X}}.
\end{equation}
By the properties of each kernel function $k_{\lambda\beta}$, the operators inside the $\mathbb{X}$-norm in \eqref{mainPV} are 
boundary-to-boundary singular integral operators of ``chord-dot-normal'' type (in the sense of \cite[Theorem~5.2.2, pp.\,623-641]{GHA.IV}).
Bearing this in mind, it follows from \cite[Theorem~2.77, item (1)]{MMM.HOPT} and \eqref{eq:EQUIV-CLASS.NNN} that there exists a constant $C\in(0,\infty)$ 
depending only on the dimension $n$, the Ahlfors regularity character of $\partial\Omega$, the space $\mathbb{X}$, and $m$ such that the following 
geometrically sensitive operator norm estimate holds:
\begin{equation}\label{Asddgggsdf}
\big\|[T]\big\|_{{\rm HWA}_{m-1}[\mathbb{X}]/{\sim}\to{\rm HWA}_{m-1}[\mathbb{X}]/{\sim}}\leq C\,\|\nu\|_{{\rm BMO}(\partial\Omega,\sigma)}
\ln\bigg(\frac{e}{\|\nu\|_{{\rm BMO}(\partial\Omega,\sigma)}}\bigg).
\end{equation}
Since $\big({\rm HWA}_{m-1}\big[{\mathbb{X}}\big]\big/\sim\,,\,\|\cdot\|_{{\rm HWA}_{m-1}[{\mathbb{X}}]/\sim}\big)$ is a Banach 
space (cf.\,Theorem~\ref{thm:HWA-X}), the claim in \eqref{Adwsfgg} follows from \eqref{Asddgggsdf}, the smallness of 
$\|\nu\|_{{\rm BMO}(\partial\Omega,\sigma)}$, and a Neumann series argument.

Granted \eqref{Adwsfgg}, we can now prove \eqref{Ttt-ISOMORPH}. To establish that $\tfrac{1}{2}I+T$ is injective, suppose that 
$\dot{h}\in{\rm HWA}_{m-1}[\mathbb{X}]$ is such that $\big(\tfrac{1}{2}I+T\big)\dot{h}=0$. Passing to the quotient space gives 
that $\big[\tfrac{1}{2}I+T\big][\dot{h}]=[0]$, and \eqref{Adwsfgg} then implies $[\dot{h}]=[0]$. Theorem~\ref{thm:HWA-X} forces 
$\dot{h}\in\dot{\mathcal{P}}_{m-2}$, at which point \eqref{polynfujdfh} permits us to write $\dot{h}=\big(\tfrac{1}{2}I+T\big)\dot{h}=0$, as wanted.

To check that $\tfrac{1}{2}I+T$ is surjective, fix $\dot{h}\in{\rm HWA}_{m-1}[\mathbb{X}]$ and use \eqref{Adwsfgg} to conclude 
that there exists $\dot{h}_{0}\in{\rm HWA}_{m-1}[\mathbb{X}]$ such that $\big[\tfrac{1}{2}I+T\big][\dot{h}_{0}]=[\dot{h}]$. 
By Theorem~\ref{thm:HWA-X} there exists $\dot{P}\in\dot{\mathcal{P}}_{m-2}$ such that 
$\big(\tfrac{1}{2}I+T\big)\dot{h}_{0}+\dot{P}=\dot{h}$. Since $\big(\tfrac{1}{2}I+T\big)\dot{P}=\dot{P}$ by \eqref{polynfujdfh}, 
we conclude that $\big(\tfrac{1}{2}I+T\big)\big(\dot{h}_{0}+\dot{P}\big)=\dot{h}$, so $\dot{h}$ belongs to the range of 
$\tfrac{1}{2}I+T$. 

We now show that the $\mathbb{X}$-Dirichlet Problem for $\Delta^{m}$ in $\Omega$ formulated in \eqref{tk1-acxvtru-INTRO.bis} 
has at least one solution satisfying items \ref{property1-INTRO}-\ref{property3-INTRO}.  
To set the stage, pick a boundary datum $\dot{f}\in{\rm HWA}_{m-1}[\mathbb{X}]$. Then \eqref{Ttt-ISOMORPH} guarantees the existence 
of a unique array $\dot{g}\in{\rm HWA}_{m-1}[\mathbb{X}]$ such that $\big(\tfrac{1}{2}I+T\big)\dot{g}=\dot{f}$, while 
\eqref{Adwsfgg} and Theorem~\ref{thm:HWA-X} (keeping in mind \eqref{eq:EQUIV-CLASS.NNN}) ensure that
\begin{equation}\label{eqqqbcbcbcnormms}
\|\dot{g}\|_{{\rm HWA}_{m-1}[\mathbb{X}]}\approx\|\dot{f}\|_{{\rm HWA}_{m-1}[\mathbb{X}]}.    
\end{equation}
Pressing on, define $u:=\dot{\mathfrak{D}}_{\Delta^{m}}\,\dot{g}$ in $\Omega$. It follows from $\eqref{modified-multiprojsdg}$ that 
\begin{equation}\label{Nsafgsdf}
u\in\mathscr{C}^{\infty}(\Omega)\,\,\text{ and }\,\, \Delta^{m}u=0\,\,\text{in }\Omega.
\end{equation}
The argument in \eqref{ddgdsdsdg}-\eqref{fwergerg} gives $\mathcal{N}_{\kappa}(\nabla^{m-1}u)\in\mathbb{X}$ with 
$\big\|\mathcal{N}_{\kappa}(\nabla^{m-1}u)\big\|_{\mathbb{X}}\lesssim\,\|\dot{g}\|_{{\rm HWA}_{m-1}[\mathbb{X}]}$, hence
\begin{equation}\label{Mainnontang}
\big\|\mathcal{N}_{\kappa}(\nabla^{m-1}u)\big\|_{\mathbb{X}}\leq C\,\|\dot{f}\|_{{\rm HWA}_{m-1}[\mathbb{X}]},   
\end{equation}
for some $C\in(0,\infty)$ independent of $\dot{f}$, and by \eqref{TTTbchbcbgg} we have
\begin{equation}\label{Ssdddtrace}
{\rm Tr}^{{}^{\kappa\text{\rm-nt}}}_{m-1}u
={\rm Tr}^{{}^{\kappa\text{\rm-nt}}}_{m-1}\big(\dot{\mathfrak{D}}_{\Delta^{m}}\dot{g}\big)
=\big(\tfrac{1}{2}I+T\big)\dot{g}=\dot{f}.
\end{equation}
Thus, $u$ is a solution to \eqref{tk1-acxvtru-INTRO.bis}, it satisfies \ref{property1-INTRO} due to \eqref{Mainnontang}, 
it satisfies \ref{property2-INTRO} by construction and \eqref{eqqqbcbcbcnormms}, and it satisfies \ref{property3-INTRO} thanks to Theorem~\ref{Jksjadfhsgh}.

At this stage, we defer the proof of uniqueness for the solution of \eqref{tk1-acxvtru-INTRO.bis} and instead focus on the solvability of 
the $\mathbb{X}_{1}$-Regularity Problem \eqref{tk1-acxvtru-INTRO.bis.REG}. With this aim in mind, fix $\dot{h}\in{\rm HWA}_{m-1}[\mathbb{X}_{1}]$ 
and set $w:=\dot{\mathfrak{D}}_{\Delta^{m}}\dot{h}$ in $\Omega$. Since ${\rm HWA}_{m-1}[\mathbb{X}_{1}]\subseteq{\rm HWA}_{m-1}[\mathbb{X}]$, 
the results in \eqref{ddgdsdsdg}-\eqref{fwergerg} are also true in the present setting. Moreover, taking an additional derivative of 
$\nabla^{m-1}w$ in \eqref{ddgdsdsdg} and using \cite[Theorem~2.76, item~(4)]{MMM.HOPT} (while bearing in mind \eqref{HShagdfsdg-THM} 
and \cite[Example~5.1.7, pp.\,556-557]{GHA.IV}) implies that
\begin{equation}\label{Ttilde-Nontangnbdds}
\mathcal{N}_{\kappa}(\nabla^{m}w)\in\mathbb{X}\,\text{ and }\,\big\|\mathcal{N}_{\kappa}(\nabla^{m}w)\big\|_{\mathbb{X}}
\leq C\,\|\dot{h}\|_{{\rm HWA}_{m-1}[\mathbb{X}_{1}]}
\end{equation}
for a constant $C\in(0,\infty)$ independent of $\dot{h}$. Granted \eqref{fwergerg} and \eqref{Ttilde-Nontangnbdds}, item {\it (2)} of Theorem~\ref{thm:HWA-X.aaa} 
implies that ${\rm Tr}^{{}^{\kappa\text{\rm-nt}}}_{m-1}w\in{\rm HWA}_{m-1}[\mathbb{X}_{1}]$ with quantitative control. 
Ultimately, these considerations show that 
\begin{equation}\label{Ttilde-def}
\begin{array}{c}
\widetilde{T}:{\rm HWA}_{m-1}[\mathbb{X}_{1}]\longrightarrow
{\rm HWA}_{m-1}[\mathbb{X}_{1}]
\\[4pt]
\widetilde{T}\dot{h}:={\rm Tr}^{{}^{\kappa\text{\rm-nt}}}_{m-1}\big(\dot{\mathfrak{D}}_{\Delta^{m}}\dot{h}\big)-\tfrac{1}{2}\dot{h}
\,\text{ for each }\,\dot{h}\in{\rm HWA}_{m-1}[\mathbb{X}_{1}]
\end{array}
\end{equation}
is a well-defined, linear, and bounded operator. Its design ensures that 
\begin{equation}\label{TTTbchbcbgg-REG}
\big(\tfrac{1}{2}I+\widetilde{T}\big)\dot{h}={\rm Tr}^{{}^{\kappa\text{\rm-nt}}}_{m-1}
\big(\dot{\mathfrak{D}}_{\Delta^{m}}\dot{h}\big)\,
\text{ for every }\,\dot{h}\in{\rm HWA}_{m-1}[\mathbb{X}_{1}].
\end{equation}
Comparing \eqref{Ttilde-def} with \eqref{Sfeggfdgerg} shows that ${\rm HWA}_{m-1}[\mathbb{X}_{1}]$ is an invariant subspace of $T$ 
and, in fact, $\widetilde{T}$ is the restriction of the operator $T$ to this subspace. 

The claim we make in relation to \eqref{Ttilde-def} is that matters may be arranged, upon further decreasing the threshold $\delta\in(0,\delta_{\ast}]$ 
if necessary, so that 
\begin{equation}\label{Ttt-ISOMORPH-REG}
\tfrac{1}{2}I+\widetilde{T}:{\rm HWA}_{m-1}[\mathbb{X}_{1}]\longrightarrow{\rm HWA}_{m-1}[\mathbb{X}_{1}]\,\,\text{ is an isomorphism}.
\end{equation}
To check \eqref{Ttt-ISOMORPH-REG}, start by noticing that $\dot{\mathcal{P}}_{m-2}\subseteq{\rm HWA}_{m-1}[\mathbb{X}_{1}]$,
hence \eqref{polynfujdfh} holds with $\widetilde{T}$ in place of $T$. Then Theorem~\ref{thm:HWA-X} ensures that the quotient 
operator $[\widetilde{T}]$ acting on ${\rm HWA}_{m-1}[\mathbb{X}_{1}]/\!\sim$ via $[\widetilde{T}][\dot{h}]:=[\widetilde{T}\dot{h}]$ 
is well defined, linear and bounded. We shall show that if $\|\nu\|_{{\rm BMO}(\partial\Omega,\sigma)}<\delta$ with $\delta\in(0,\delta_\ast]$ 
small enough then the operator
\begin{equation}\label{Ttilde-iso}
\tfrac{1}{2}I+[\widetilde{T}\,]:{\rm HWA}_{m-1}[\mathbb{X}_{1}]\big/\sim\,\rightarrow\,
{\rm HWA}_{m-1}[\mathbb{X}_{1}]\big/\sim\,\text{ is a Banach-space isomorphism}.
\end{equation}
To see why this is the case, for each $\dot{h}\in{\rm HWA}_{m-1}[\mathbb{X}_1]$ use \eqref{Ttilde-def} and \eqref{eq:EQUIV-CLASS.NNN} with 
${\mathbb{X}}_1$ in  place of ${\mathbb{X}}$ to write
\begin{equation}\label{mainquotnorm--nn}
\Big\|[\widetilde{T}\,][\dot{h}]\Big\|_{{\rm HWA}_{m-1}[\mathbb{X}_1]/{\sim}}
=\sum_{|\lambda|=m-1}\Big\|\big[\partial^{\lambda}\big(\dot{\mathfrak{D}}_{\Delta^{m}}\dot{h}\big)\big]
\big|_{\partial\Omega}^{{}^{\kappa\text{\rm -nt}}}-\tfrac{1}{2}h_{\lambda}\Big\|_{{\mathbb{X}}_1}.
\end{equation}
As in the passage from \eqref{mainquotnorm} to \eqref{mainPV}, the right-hand side of \eqref{mainquotnorm--nn} is bounded above by 
\begin{equation}\label{mainPV--nn}
\sum_{|\beta|=|\lambda|=m-1}\bigg\|\lim_{\varepsilon\to 0^{+}}\int_{\substack{y\in\partial\Omega\\|\cdot-y|>\varepsilon}}
\nu(y)\cdot(y-\cdot)\,k_{\lambda\beta}(\cdot-y)h_{\beta}(y)\,d\sigma(y)\bigg\|_{{\mathbb{X}}_1}.
\end{equation}
Since according to \cite[Theorem~2.77, item (2)]{MMM.HOPT} boundary-to-boundary singular integral operators of ``chord-dot-normal'' type also satisfy 
geometrically sensitive operator norm estimates on the Sobolev space ${\mathbb{X}}_1$, this analysis shows that there exists a constant $C\in(0,\infty)$ 
depending only on the dimension $n$, the Ahlfors regularity character of $\partial\Omega$, the space $\mathbb{X}$, and $m$ such that 
\begin{equation}\label{MNnfsdghsdf}
\big\|[\widetilde{T}\,]\big\|_{{\rm HWA}_{m-1}[\mathbb{X}_1]/{\sim}\to{\rm HWA}_{m-1}[\mathbb{X}_1]/{\sim}}
\leq C\,\|\nu\|_{{\rm BMO}(\partial\Omega,\sigma)}\ln\bigg(\frac{e}{\|\nu\|_{{\rm BMO}(\partial\Omega,\sigma)}}\bigg).
\end{equation}
Given that $\big({\rm HWA}_{m-1}\big[{\mathbb{X}}_1\big]\big/\!\sim\,,\,\|\cdot\|_{{\rm HWA}_{m-1}[{\mathbb{X}}_1]/\sim}\big)$ is a Banach space 
(cf. Theorem~\ref{thm:HWA-X}), it follows from \eqref{MNnfsdghsdf}, the smallness of $\|\nu\|_{{\rm BMO}(\partial\Omega,\sigma)}$, and 
a Neumann series argument that \eqref{Ttilde-iso} holds. With this in hand, \eqref{Ttt-ISOMORPH-REG} follows via an argument similar 
to the end-game in the proof of \eqref{Ttt-ISOMORPH}.

We finally turn to the task of actually solving the $\mathbb{X}_{1}$-Regularity Problem \eqref{tk1-acxvtru-INTRO.bis.REG} for an 
arbitrary boundary datum $\dot{f}\in{\rm HWA}_{m-1}[\mathbb{X}_{1}]$. From \eqref{Ttt-ISOMORPH-REG} we know that there
exists a unique array $\dot{g}\in{\rm HWA}_{m-1}[\mathbb{X}_{1}]$ such that $\big(\tfrac{1}{2}I+\widetilde{T}\big)\dot{g}=\dot{f}$, 
and which is quantitatively controlled by $\dot{f}$, i.e., $\|\dot{g}\|_{{\rm HWA}_{m-1}[\mathbb{X}_{1}]}\approx\|\dot{f}\|_{{\rm HWA}_{m-1}[\mathbb{X}_{1}]}$.
If we now define $u:=\dot{\mathfrak{D}}_{\Delta^{m}}\,\dot{g}$ in $\Omega$, then \eqref{Nsafgsdf} holds, \eqref{TTTbchbcbgg-REG} and the choice of $\dot{g}$ 
ensure that ${\rm Tr}^{{}^{\kappa\text{\rm-nt}}}_{m-1}u=\dot{f}$, and \eqref{fwergerg} together with \eqref{Ttilde-Nontangnbdds} implies
\begin{equation}\label{Ssdgfdregest}
\big\|\mathcal{N}_{\kappa}(\nabla^{m-1}u)\big\|_{\mathbb{X}}+\big\|\mathcal{N}_{\kappa}(\nabla^{m}u)\big\|_{\mathbb{X}}
\leq C\,\|\dot{f}\|_{{\rm HWA}_{m-1}[\mathbb{X}_{1}]},
\end{equation}
for some $C\in(0,\infty)$ independent of $\dot{f}$. In conclusion, $u$ is a solution to the $\mathbb{X}_{1}$-Regularity Problem 
\eqref{tk1-acxvtru-INTRO.bis.REG}, it satisfies the estimate in \eqref{a538bxd9-tk2-INTRO.REG}, and may be expressed as in 
\eqref{eq:U-INT-REP-FORM} for some $\dot{g}\in{\rm HWA}_{m-1}[\mathbb{X}_{1}]$ uniquely determined and quantitatively controlled by $\dot{f}$.

Granted the solvability of the $\mathbb{X}_{1}$-Regularity Problem for each ${\mathbb{X}}$ in the hierarchy of good spaces on $\partial\Omega$,  
we may invoke \cite[Corollary~18.2]{MMM.HOPT} to conclude the uniqueness of the $\mathbb{X}$-Dirichlet Problem \eqref{tk1-acxvtru-INTRO.bis}.
Ultimately, this shows that the $\mathbb{X}$-Dirichlet Problem for $\Delta^m$ in $\Omega$ is well posed. As a consequence, 
the $\mathbb{X}_{1}$-Regularity Problem \eqref{tk1-acxvtru-INTRO.bis.REG} is also well posed.

It remains to justify \eqref{a538bxd9-tk9-INTRO}. In one direction, let $u$ solve the $\mathbb{X}$-Dirichlet Problem \eqref{tk1-acxvtru-INTRO.bis}
with boundary datum $\dot{f}\in{\rm HWA}_{m-1}[\mathbb{X}_{1}]$. By the well-posedness of the $\mathbb{X}_{1}$-Regularity Problem and the uniqueness in the 
$\mathbb{X}$-Dirichlet Problem, $u$ must be the unique solution of the $\mathbb{X}_{1}$-Regularity Problem \eqref{tk1-acxvtru-INTRO.bis.REG}
with boundary datum $\dot{f}\in{\rm HWA}_{m-1}[\mathbb{X}_{1}]$. In particular, $\mathcal{N}_{\kappa}(\nabla^{m}u)\in\mathbb{X}$ with quantitative control. 
Conversely, assume that $u$ is the solution of the $\mathbb{X}$-Dirichlet Problem 
in \eqref{tk1-acxvtru-INTRO.bis} with boundary datum $\dot{f}\in{\rm HWA}_{m-1}[\mathbb{X}]$, and assume further that 
$\mathcal{N}_{\kappa}(\nabla^{m}u)\in\mathbb{X}$. Then item~{\it (2)} in Theorem~\ref{thm:HWA-X.aaa} yields
$\dot{f}={\rm Tr}^{{}^{\kappa\text{\rm-nt}}}_{m-1}u\in{\rm HWA}_{m-1}[\mathbb{X}_{1}]$, quantitatively.
\end{proof}


\end{document}